\documentclass[11pt,reqno]{amsart}

\usepackage{array} 
\usepackage{varwidth} 
\usepackage[math]{cellspace}
\usepackage{longtable}

\usepackage{fullpage}
\usepackage{amsmath,amsfonts,amsthm,amssymb}
\usepackage{color}
\usepackage{graphicx}
\usepackage{nicefrac}
\usepackage[colorlinks=true,backref]{hyperref}
\usepackage{cleveref}

\newcommand{\hide}[1]{}

\DeclareMathOperator{\Rep}{Rep}
\DeclareMathOperator{\Sem}{Sem}
\DeclareMathOperator{\Ising}{Ising}
\DeclareMathOperator{\Fib}{Fib}
\DeclareMathOperator{\PSO}{PSO}
\DeclareMathOperator{\PSU}{PSU}
\DeclareMathOperator{\Irr}{Irr}
\newcommand{\Irrstar}{\Irr^{\prime}\!}
\DeclareMathOperator{\Hom}{Hom}
\DeclareMathOperator{\FPdim}{FPdim}
\DeclareMathOperator{\so}{\mathfrak{so}}
\DeclareMathOperator{\Trace}{Tr}
\newcommand{\Tr}[1]{\Trace\!\left(#1\right)}
\newcommand{\mcC}{\mathcal{C}}
\newcommand{\mbbC}{\mathbb{C}}
\newcommand{\mfS}{\mathfrak{S}}

\usepackage{braket}
\usepackage{bbm}
\usepackage{float}
\newcommand{\one}{\mathbbm{1}}

\usepackage{tikz}
\usetikzlibrary{tikzmark}
\usetikzlibrary{arrows.meta}
\usepackage{subcaption}

\newtheorem{thm}{Theorem}
\newtheorem*{thm*}{Theorem}

\newtheorem{Acor}{Algebraic Corollary}

\newtheorem*{rem}{Remark}
\newtheorem*{rmk}{Remark}

\newtheorem{lem}{Lemma}

\newtheorem{obs}{Observation}

\newcommand{\mc}{\mathcal}

\newcommand{\N}{\mathbb{N}}
\newcommand{\Z}{\mathbb{Z}}
\newcommand{\mbbZ}{\mathbb{Z}}
\newcommand{\mbbN}{\mathbb{N}}

\renewcommand{\t}{\tau}
\newcommand{\paren}[1]{\left(#1\right)}
\renewcommand{\d}{\delta}
\newcommand{\lcb}{\left\{}
\newcommand{\rcb}{\right\}}
\newcommand{\hs}{\hspace{1mm}}

\newcommand{\size}[1]{\left| #1 \right|}

\newcommand{\defn}[1]{\textit{#1}}

\renewcommand{\emptyset}{\varnothing}

\definecolor{CB_purple}{RGB}{51,24,136}
\definecolor{CB_green}{RGB}{17,119,51}
\definecolor{CB_lightblue}{RGB}{136,204,238}
\definecolor{CB_red}{RGB}{136,34,85}

\title{Hypergraph Characterization of Fusion Rings}

\author{Paul Bruillard$^\ast$}

\author{Kathleen Nowak$^\ast$}

\author{Stephen J.\ Young$^\dagger$}
\address{Pacific Northwest National Laboratory, Richland, WA, 99352}
\email{stephen.young@pnnl.gov}
\thanks{$^\ast$Work performed while at Pacific Northwest National Laboratory. $^\dagger$Corresponding author.}

\begin{document}
\newcolumntype{M}{>{\begin{varwidth}{.3\textwidth}}Sc<{\end{varwidth}}}
\begin{abstract}
 We present a correspondence between multiplicity-free, self-dual,
 fusion rings and a digraph, hypergraph pair $(D,H)$.  This
 correspondence is used to provide a complete characterization of all
 fusion rings corresponding to graphical properties of $D$.  Further,
 we exploit this correspondence to provide a complete list of all
 non-isomorphic,  
 self-dual, multiplicity-free fusion rings of rank at most 8.  
\end{abstract}
\maketitle


\section{Introduction}
  Fusion categories arise in a number of settings in both mathematics and
  physics. In physics, they are part of the algebraic data necessary to describe
  a topological quantum field theory \cite{W1}; in computer science, they
  aid in understanding the properties of certain quantum computers
  \cite{Freedman03,NR1}; in mathematics, they axiomatize representation
  theory \cite{ENO1,DGNO1}.
  Mathematically, fusion categories naturally arise from the study of
  representations of finite groups, quantum groups, and certain Hopf algebras
  \cite{ENO1,DGNO1}. Fusion categories are central to many conjectures found in
  the literature such as the Property-F Conjecture \cite{NR1}, Gauging
  Conjecture \cite{ACRW1}, Weakly Group-Theoretical Conjecture \cite{ENO2}, and
  Kaplansky's Sixth Conjecture \cite{AIMPL}.

  Fusion categories are often stratified by their \textit{rank}, the number of
  isomorphism classes of simple objects. However, even for fixed rank fusion
  categories are difficult to fully characterize. At the time this research was initially  presented~\cite{fusion_ring_talk}, a full
  classification was only known through rank 3, though partial results were available up through rank 5 \cite{O2,O3,B1,BO1,L1}. Recently, there has been significant progress, see for instance the thesis of Vercleyen \cite{vercleyen2024lowrankmultiplicityfreefusioncategories} which provides a complete list of all multiplicity-free fusion categories up to rank 7 and was later extended to rank 9~\cite{Vercleyen2022OnLR}, \cite{alekseyev2025classifyingintegralgrothendieckrings} which classifies all modular data of integral modular fusion categories up to rank 13, and \cite{Grothendieck_rank6} which classifies the Grothendieck rings for complex fusion categories of multiplicity one up to rank 6.
  
  Part of the challenge of the underlying
  classification problem is a lack of structure and enumeration tools. One
  partial solution is to characterize the Grothendieck rings of fusion
  categories rather than the categories themselves. This alleviates certain
  burdens associated with categorification and associativity, however the
  problem still remains difficult. In fact, at the time this research~\cite{fusion_ring_talk}, there were no known algorithms to enumerate all candidate rings. Recently, Vercleyen and Slingerland~\cite{Vercleyen2022OnLR} have been able to generate an exhaustive list of fusion rings for several multiplicities up to rank 9 by using clever search heuristics to reduce the search space. While their computational results subsume the computational results described in this work, we believe that the novel correspondence between fusion rings and certain digraph/hypergraph pairs provides independent value. 
  Indeed, as we illustrate in Section \ref{S:undirected}, this correspondence 
  will enable us to apply tools from graph theory and combinatorics, fields
  accustomed to handling mathematical constructs with little structure, to the enumeration problem significantly reducing the enumeration space.  Our
  techniques will allow us to completely characterize certain multiplicity-free
  self-dual fusion rings and are extendable to cases with
  more complex duality structures and rings with multiplicity. Our main result
  is:

  \begin{thm*}
    The multiplicity-free, self-dual braided fusion categories generated by an
    undirected, triangle-free graph $G$ are the Grothendieck equivalent to one
    of the following:
    \begin{enumerate}
      \item $\Fib$
      \item $PSU(3)_2$,
      \item $PSU(2)_6$, and 
      \item $Rep(G)$ where $G$ is an elementary abelian 2-group of order $2^k$. 
    \end{enumerate}
  \end{thm*}

  The organization of the paper is as follows. In the next section, we introduce
  the notation that will be used throughout and recall the necessary definitions
  for algebraic topology and structural graph theory. Then in Section
   \ref{S:correspond} we introduce a correspondence between fusion rings and
  certain digraph/hypergraph pairs. From there we
  begin to classify the subclasses of graphs which, by this correspondence,
  yield fusion rings. Our main contribution is presented in
  Section \ref{S:undirected} where we provide a classification of braided fusion
  categories up to Grothendieck equivalence arising from triangle free
  undirected graphs.  In Section \ref{S:lowrank} we illustrate how the
  correspondence developed in this work can be used to facilitate the
  complete enumeration of fusion rings, in particular, we give a
  complete listing of all multiplicity-free, self-dual fusion rings up
  to rank 8.

\section{Preliminaries}
  \subsection{Fusion Rings}   Although our focus in this paper is the
classification and enumeration of fusion rings, this study is
motivated by the connections of fusion categories to both physics and
mathematics.  Thus, we briefly outline here the definitions associated
with fusion categories and refer the interested reader to the
excellent monographs \cite{ENO1,DGNO1} and references therein.

    A \defn{fusion category} over $\mbbC$, $\mcC$, is a $\mbbC$-linear
    semisimple rigid monoidal category with finitely many simple objects and
    finite dimensional spaces of morphisms \cite{O2}. The fusion
    category is said to be \defn{braided} if the monoidal product, $\otimes$, is
    commutative up to isomorphism with that isomorphism subject to certain
    consistency conditions \cite{Bki}. We will
    denote the set of isomorphism classes of simple objects in $\mcC$ by
    $\Irr\paren{\mcC}$ and the isomorphism classes of simple objects themselves
    by $X_{a}$. These simple objects will always be indexed from $0$ and ordered such
    that $X_{0}=\one$ is the monoidal unit.  For convenience of notation we will denote by $\Irrstar\paren{\mcC}$ as the non-identity elements of $\Irr\paren{\mcC}$, that is $\Irrstar\paren{\mcC} = \Irr\paren{C}\backslash\one$. The dual of an object, $X$, is
    denoted by $X^{*}$ and the category is said to be \defn{self-dual} if
    $X_{a}=X_{a}^{*}$ for all $X_{a}\in\Irr\paren{\mcC}$. Duality induces an
    involution on $\Irr\paren{\mcC}$ as well as the index set of $\Irr\paren{\mcC}$, specifically $a^{*}$
    is defined by the relationship $X_{a^{*}}=\paren{X_{a}}^{*}$. The number of isomorphism
    classes of simple objects is referred to as the \defn{rank} of $\mcC$. The
    dimensions of the hom-spaces define the \defn{fusion matrices}, $N_{a}$,
    by
    $\paren{N_{a}}_{b,c}=N_{a,b}^{c}=\dim_{\mbbC}\Hom_{\mcC}\paren{X_{a}\otimes
    X_{b},X_{c}}$. The fusion matrices are normal
    non-negative matrices and hence are subject to the Frobenius-Perron Theorem.
    The Frobenius-Perron eigenvalue of $N_{a}$ is called the
    \defn{Frobenius-Perron dimension} or \defn{FP-dimension} of $X_{a}$ and
    is denoted by $\FPdim\paren{X_{a}}$ or more succinctly as $d_{a}$. The
    \defn{FP-dimension of $\mcC$} is defined by
    $\FPdim\paren{\mcC}=\sum_{a=0}^{r-1}d_{a}^{2}$, where $r$ is the
    rank of the fusion category $\mcC$. 

    A good deal of information about a fusion category is captured by its
    Grothendieck ring, that is the ring formed by the equivalence
    classes of the objects of the category together with a certain
    class of relations
    defined by short-exact sequences in the category. Such rings can be studied in the absence of their
    categorical origins \cite{O1}. To this end, recall from \cite{O1}, a
    \defn{$\mbbZ_{+}$-basis} of an algebra, which is free as a module over $\mbbZ$, is a basis $B=\lcb
    X_{i}\rcb$ such that $X_{i}X_{j}=\sum_{k}N_{i,j}^{k}X_{k}$ for some
    $N_{i,j}^{k}\in\mbbN$. In other words, the product of basis
    elements can be expressed as a the linear combination of the
    elements of the basis.  The elements of the basis are called \defn{simple},
    the size of the basis is called the \defn{rank}, and the
    coefficients describing the multiplication (${ N_{ij}^k }$) are
    called the \defn{structure coefficients}. A
    \defn{$\mbbZ_{+}$-ring} is a unital algebra over $\mbbZ$ with a fixed
    $\mbbZ_{+}$-basis. Finally, a \defn{unital based ring} is a
    $\mbbZ_{+}$-ring, $A$, with a unit $\one$, a basis $X=\lcb X_{i}\rcb_{i\in
    I}$ with unit $\one = X_{0}$, such that there exists an involution $\cdot^*:X\to X$ (as well as on $I$, as above) and a group homomorphism
    $\t:A\to\mbbZ$ which satisfy that:
    \begin{itemize}
      \item[(i)] The induced map $\sum_{i\in I}a_{i}X_{i}\mapsto \sum_{i\in
      I}a_{i}X^*_{i}$ with $a_{i}\in \mbbZ$ is an anti-involution and 
      \item[(ii)] For all $i,j \in I$, $\t\paren{X_{i}X_{j}}=\d_{i,j^*}$.
      \end{itemize}
      
    The Grothendieck ring of every (braided) fusion
    category yields a (commutative) unital based ring. For this reason we will
    often refer to unital based rings as \defn{fusion rings}.
    \begin{rmk}
      Henceforth we will only be interested in commutative fusion rings and so
      we will drop the word commutative and refer only to fusion rings.
    \end{rmk}
   
    A series of direct calculation reveals that the structure coefficients of a
    fusion ring are highly symmetric. For example, we note that 
    \[ \tau((X_iX_j)X_{k^*}) = \tau(\sum_{\ell}N_{ij}^{\ell}X_{\ell}X_{k^*}) = \sum_{\ell} N_{ij}^{\ell} \delta_{\ell k} = N_{ij}^k\]
    and 
    \[ \tau((X_iX_j)X_{k^*}) = \tau(X_i(X_jX_{k^*})) = \tau(\sum_{\ell}N_{jk^*}^{\ell}X_i X_{\ell}) = \sum_{\ell} N_{jk^*}^{\ell} \tau(X_i X_{\ell})= N_{jk^*}^{i^*}.\]
    Thus $N_{ij}^k = N_{jk^*}^{i^*}$, which is often referred to
    as \defn{rigidity} of the fusion ring.  As a consequence, we have that $N_{ij}^0 = N_{j0}^{i^*} = \delta_{i^*j}$.  Furthermore, since the
    duality operation is an anti-involution,
   \[ \sum_{k} N_{ij}^k X_{k^*} = \paren{\sum_k N_{ij}^k X_k}^* = \paren{X_iX_j}^* = 
    X_{j^*}X_{i^*} = \sum_{k^*} N_{j^*i^*}^{k^*} X_{k^*},\] and in
  particular  $N_{ij}^k = N_{j^*i^*}^{k^*}$ for all $i$, $j$, and $k$.  Combining these symmetries, we have that 
    \[ N_{ij}^{k} = N_{jk^*}^{i^*} = N_{k^*i}^{j^*} = N_{i^*k}^j = N_{kj^*}^{i} = N_{j^*i^*}^{k^*}.\] 
    In fact, together with the commutativity relationship $N_{ij}^k = N_{ji}^k$, these symmetries are necessary and sufficient for the $N_{ij}^{k}$
    to yield a fusion ring. This allows us to dispense with the homomorphism
    $\t$, and consider fusion rings as $\mbbZ$-algebras with presentation 
    \begin{align*}
      \mc{R} := \left\langle X_0 = \one, X_1, \dots, X_{r-1} \mid X_iX_j = \sum_{k = 0}^{r-1} N_{ij}^kX_k, \hs \one X_i = X_i\right\rangle
    \end{align*}
    where there is an involution $*$ with $N_{ii^*}^0 = 1$ and the structure coefficients satisfy the three symmetry
    properties (rigidity, self-involution, and commutativity).    In fact, not only do the structure coefficients completely determine the
    fusion ring, they provide a faithful representation of it.  In particular, letting $a$ and $b$ be arbitrary indices, we have that 
    \begin{align*}
       e_a^T N_i N_j e_b &= \sum_{\ell} N_{ia}^{\ell} N_{j\ell}^{b} 
       = \sum_{\ell} N_{ia}^{\ell}N_{b^*j}^{\ell*} 
       = \sum_{s,t} N_{ia}^sN_{b^*j}^t \delta_{st^*} 
       = \sum_{s,t} N_{ia}^sN_{b^*j}^t \tau(X_s X_t)  \\
       &= \tau\paren{\sum_{s,t} N_{ia}^sN_{b^*j}^t X_s X_t} 
       =\tau\paren{ \paren{\sum_s N_{ia}^s X_s} \paren{\sum_t N_{b^*j}^t X_t}}
       = \tau( X_i X_a X_{b^*} X_j ) \\
    \end{align*}
    and by the commutativity of the fusion ring, this is equal to
    \begin{align*}
     \tau( X_i X_j X_a X_{b^*}) &= \tau\paren{\paren{\sum_{s} N_{ij}^sX_s}\paren{\sum_t N_{ab^*}^t X_t}} = \tau\paren{\sum_{s,t} N_{ij}^s N_{ab^*}^t X_sX_t} 
       = \sum_{s,t} N_{ij}^s N_{ab^*}^t \tau(X_sX_t) \\
       &= \sum_k N_{ij}^k N_{ab^*}^{k^*} 
       = \sum_k N_{ij}^k N_{ka}^b 
       = \sum_k N_{ij}^k e_a^TN_k e_b.
    \end{align*}
     As a consequence, 
    $N_{j}N_{i}=N_{i}N_{j}=\sum_{k=0}^{r-1}N_{i,j}^{k}N_{k}$ and the fusion matrices provide a faithful representation of the fusion ring.

  \subsection{Graphs and Hypergraphs}
     While our primary motivation is to study fusion rings, we will do so through
  the lens of graph theory. For this reason
  we will take a short detour and define the essential terminology for graphs and
  hypergraphs. Once this is accomplished we will establish a correspondence
  between self-dual, multiplicity-free fusion rings and graph/hypergraph pairs.
  
    A (simple) graph $G$ is a ordered pair $(V,E)$ where $V$ is a finite set of \defn{vertices} and $E$ is a subset of the 2-elements subsets of $V$ referred to as the \defn{edges}. We say two vertices are \emph{adjacent}, denoted $u \sim v$, if 
    $\{u,v\} \in E$. A graph is said to be \defn{trivial} if it contains exactly one vertex and \defn{empty} if the edge set is empty.  For a vertex $v$, we define its neighborhood as $N(v) = \{ u \mid \{u,v\} \in E\}$.  The cardinality of the neighborhood of $v$ is the \defn{degree} of $v$, $\deg(v)$.  We will denote by $\delta$ and $\Delta$, the minimum and maximum degree in a graph.   

    A \defn{subgraph} of $G$ is formed by selecting a subset of vertices and
    edges of the graph. That is, $H = (V', E')$ is a subgraph of $G$ if $V'
    \subseteq V$ and $E' \subseteq E$. A \defn{vertex-induced subgraph} is
    formed by taking a subset of vertices $W$ and precisely those edges of $G$
    with both endpoints in $W$. The subgraph of $G$ induced by $W$ is denoted
    $G[W]$. A path $P$ in $G$ is a sequence of vertices $(v_0, v_1, \dots,
    v_k)$ such that one of $(v_i, v_{i+1})$ and $(v_{i+1},v_i)$ is an edge for
    all $i = 0, \dots, k-1$. We define the \defn{length} of a path $P = (v_0,
    v_1, \dots, v_k)$ to be $k$. We call a path \defn{simple} if it does not
    contain any repeated vertices. The \defn{distance} between two
    vertices $u$ and $v$, denoted $d(u,v)$, is defined to be the
    length of the shortest path from $u$ to $v$.  If no such path
    exists, the distance is defined to be infinite.  The maximum
    distance over all pairs of vertices is called the \defn{diameter}
    of a graph.  A \defn{connected component} (or component for
    short) of $G$ is a maximal subgraph $H$ in which any two vertices are connected by
    a path. We say that $G$ is \defn{connected} if it contains precisely one
    component.   We will say that a vertex is \emph{isolated} if its connected
    component has precisely one vertex.

    Similarly to a graph, a \emph{directed graph} or \emph{digraph}, is an ordered pair $(V,E)$ where $E$ is a subset of $V \times V$.  Much of the terminology above extends naturally to digraphs with additional consideration taken for the direction of the edges.  For example, instead of the neighborhood of a vertex $v$ we consider the \defn{in-} and \defn{out-neighborhoods} of a vertex, $N^-(v) = \{ u \mid (u,v) \in E\}$ and $N^+(v) = \{u \mid (v,u) \in E\}$, respectively.  The \defn{in-} and \defn{out-degrees} of a vertex are defined analogously and denoted $\deg^-(v)$ and $\deg^+(v)$.  An
    edge of the form $(v,v)$ is called a \defn{loop} and we will say a vertex is
    \emph{trivial} if it is an isolated vertex without a self-loop.

    A \defn{hypergraph} is a generalization of an undirected graph where the
    edges can join, not just pairs, but any number of vertices. Formally, a
    hypergraph is a pair $H = (V,E)$ where $V$ is a set of vertices and $E
    \subset \mc{P}(V)$ is a set of nonempty subsets of $V$. The elements of $E$
    are called \defn{hyperedges}. In the special case where all the hyperedges
    have size $k$, we say that $H$ is \defn{$k$-uniform}. In particular, a
    simple graph is a $2$-uniform hypergraph.  For further graph
    theoretic terms we refer the reader to \cite{MGT}.

\section{Fusion Ring/Digraph Correspondence}\label{S:correspond}
  In this section we will establish a correspondence between a certain
  class of fusion rings which are multiplicity-free (i.e., $N_{ij}^k
  \leq 1$ for all $i,j,k$) and self-dual (i.e., the involution
  associated with the fusion ring is the identity map) and 
  certain digraph/hypergraph pairs. This correspondence will enable the majority
  of our study.

  We note that the restriction to considering self-dual
  fusion rings induces additional symmetries on the structure
  coefficients, in particular,
  \[ N_{ij}^k = N_{ji}^k = N_{ik}^j = N_{ki}^j = N_{jk}^i = N_{kj}^i.\]
  Thus, it is easy to see that for fixed rank $r$ all multiplicity-free,
  self-dual fusion rings can be determined by considering all
  $\binom{r}{3} + 2\binom{r}{2} + r$ unique possible entries in the fusion
  matrices.  However, this approach has several implicit inefficiencies
  that significantly limit its effectiveness.  Specifically, it fails to
  account for the dependencies between the entries resulting from the
  commutation of the fusion matrices and it fails to account for the
  $\sim (r-1)!$ isomorphic copies of each assignment resulting from the
  relabeling of the elements of the basis.  By rephrasing the problem
  of guessing entries of the fusion matrices as guessing a digraph and a
  hypergraph we can partially address both these issues. This is accomplished by
  exploiting the commutation relationship and exploiting extant tools for dealing
  with graph isomorphism~\cite{nauty}.

  More concretely, we note that there are $2\binom{r}{2} + r$ possible
  edges in a $r$-vertex digraph with self-loops and $\binom{r}{3}$ possible edges
  in a 3-uniform hypergraph on $r$ vertices.  Thus we can build a
  representation of all potential fusion matrices by specifying the
  edges of a directed graph, $D$, with self-loops on $\{0,1,\ldots,r-1\}$ and the edges of a
  3-uniform hypergraph, $H$, on the same set.  Specifically, the digraph and hypergraph
  encode the fusion rules as follows: 
  \begin{itemize}
      \item for $i \in
  \{0,1,\ldots,r-1\}$, there is a self-loop at $i$ in $D$ if and only if $N_{ii}^i =
  1$,
  \item  for $\{i,j\} \subseteq \{0,1,\ldots,r-1\}$ there is a directed edge $(i,j)$ in $D$ if
  and only if $N_{ii}^j = N_{ij}^i = N_{ji}^i = 1$, and
  \item for $\{i,j,k\}
  \subseteq \{0,1,\ldots,r-1\}$ there is a hyperedge $\{i,j,k\}$ in $H$ if and only if
  $N_{ij}^k = N_{ji}^k = N_{ik}^j = N_{ki}^j = N_{jk}^i = N_{kj}^i =
  1$.  
  \end{itemize}   
  We note that the entries of $N_i$ complete describe the edges
  incident with the node $i$, in particular, there is a self-loop at
  $i$ if and only if $\paren{N_i}_{ii} = 1$, a directed edge
  $(i,j)$ if and only if $\paren{N_i}_{jj} = 1$, a directed edge
  $(j,i)$ if and only if $\paren{N_i}_{ij} = \paren{N_i}_{ji} = 1$,
  and a hyperedge $\{i,j,k\}$ if and only if $\paren{N_i}_{jk} =
  \paren{N_i}_{kj} = 1.$  In particular, the edges incident with $0$
  are completely determined by the convention that $X_0 = \one$ and thus $N_0 = I$.
  Specifically, the only edges incident to $0$ are a self-loop and a
  directed edge from $0$ to $i \in [r-1] = \{1,\ldots, r-1\}.$  Thus
  for the remainder, we will focus on the correspondence between a
  self-dual, multiplicity-free rank $r$ fusion ring and a
  digraph/hypergraph pair on $[r-1]$.  In order to illustrate the
  correspondence between fusion rule, fusion matrix, a graphical
  presentations of a fusion ring, in Figure \ref{F:rep}, we present all
  three presentations for the fusion ring $\Rep\paren{\mbbZ_{3}^{2}\rtimes\mbbZ_{2}}$.

  For an arbitrary digraph, hypergraph pair $(D,H)$ we
  will say that $(D,H)$ \defn{generates a fusion ring} if the associated
  matrices $\{N_i\}$ are the fusion matrices for some fusion ring.  We
  will further say that a digraph $D$ \defn{generates a fusion ring} if there
  exists some hypergraph $H$ such that the pair $(D,H)$ generates a
  fusion ring. Finally, we say that $D$ (resp. $(D,H)$) generates a fusion
  category if $D$ (resp. $(D,H)$) generates the Grothendieck ring of the
  category.

    \begin{rmk}
        It is worth noting that the graphical presentation of a self-dual multiplicity-free fusion ring fully captures the invertibility of the non-identity elements.  Specifically, if $X_i \in \Irr(\mcC)$ is invertible element in $\mcC$, then, by the self-duality, it is its own inverse, and in particular, $N_{ii}^j = 0$ for all $j \neq 0$.  As a consequence, if $X_i$ is not the identity, then $i$ does not have a self-loop or an out-neighbor in the graphical presentation of $\mcC$. The converse direction follows along similar lines.
    \end{rmk}

  \begin{figure}
  \centering
  \begin{subfigure}{\textwidth}
  \centering
  \begin{tabular}{c | c c c c c}
  $\otimes$ & $X_1$ & $X_2$ & $X_3$ & $X_4$ & $X_5$ \\ \hline
  $X_1$ &  $\mathbbm{1} + X_1 + X_5$  &  $X_3+X_4 $ & $X_2 + X_4 $ &  $X_2 + X_3 $ & $X_1 $ \\
  $X_2$ &  $X_3+X_4 $  &  $\mathbbm{1} + X_2 + X_5$ & $X_1 + X_4 $ &  $X_1 + X_3$ & $X_2 $ \\
  $X_3$ &  $X_2 + X_4 $  &  $X_1 + X_4 $ & $ \mathbbm{1} + X_3 + X_5$ &  $X_1+X_2$ & $X_3 $ \\
  $X_4$ &  $X_2 + X_3 $  &  $X_1 + X_3 $ & $X_1 + X_2 $ &  $ \mathbbm{1} + X_4 +
  X_5$ & $X_4 $ \\
  $X_5$ &  $X_1$  &  $X_2 $ & $X_3 $ &  $X_4 $ & $ \mathbbm{1} $ \\
  \end{tabular}
  \caption{Fusion Rule Presentation}
  \vspace{\baselineskip}
  \end{subfigure}
  \begin{subfigure}{\textwidth}
  \centering
  \begin{picture}(337,205)
  \put(0,50){$\left[\begin{matrix} 0 & 0 & 0 & 1 &0 &0 \\ 0 & 0 & 1 & 0 
        &1&0 \\ 0 & 1 & 0 & 0 &1 &0 \\ 1 & 0 & 0 & 1 &0 &1 \\ 0 & 1 & 1 
        & 0 &0 &0 \\ 0 & 0 & 0 & 1 &0 &0 \end{matrix}\right]$}
  \put(45,0){$X_3$}
  \put(0,160){$\left[\begin{matrix} 1 & 0 & 0 & 0 &0 &0 \\ 0 & 1 & 0 & 0 
        &0 &0 \\ 0 & 0 & 1 & 0 &0 &0 \\ 0 & 0 & 0 & 1 &0 &0 \\ 0& 0 & 0 
        & 0 &1&0 \\ 0 & 0 & 0 & 0 &0 &1 \end{matrix}\right]$}
  \put(45,110){$\mathbbm{1}$}
  \put(120,50){$\left[\begin{matrix} 0 & 0 & 0& 0 &1 &0 \\ 0 & 0 & 1 & 1 
        &0 &0 \\ 0 & 1 & 0 & 1 &0 &0 \\ 0 & 1 & 1 & 0 &0 &0 \\ 1 & 0 & 0 
        & 0 &1 &1 \\ 0 & 0 & 0 & 0 &1 &0 \end{matrix}\right]$}
  \put(165,0){$X_4$}
  \put(120,160){$\left[\begin{matrix} 0 & 1 & 0 & 0 &0 &0 \\ 1 & 1 & 0 & 0 
        &0 &1 \\ 0 & 0 & 0 & 1 &1 &0 \\ 0 & 0 & 1 & 0 &1 &0 \\ 0 & 0 & 1 
        & 1 & 0 &0 \\ 0 & 1 & 0 & 0 &0 &0 \end{matrix}\right]$}
  \put(165,110){$X_1$}
  \put(240,50){$\left[\begin{matrix} 0 & 0 & 0 & 0 &0 &1 \\ 0 & 1 & 0 & 0 
        &0 &0 \\ 0 & 0 & 1 & 0 &0 &0 \\ 0 & 0 & 0 & 1 &0 &0 \\ 0 & 0 & 0 
        & 0 &1 &0 \\ 1 & 0 & 0 & 0 &0 &0 \end{matrix}\right]$}
  \put(285,0){$X_5$}
  \put(240,160){$\left[\begin{matrix} 0 & 0 & 1 & 0 &0 &0 \\ 0 & 0 & 0 & 1 
        &1&0 \\ 1& 0 & 1 & 0 &0 &1 \\ 0 & 1 & 0 & 0 &1 &0 \\ 0 & 1 & 0 
        & 1 &0 &0 \\ 0 & 0 & 1 & 0 &0 &0 \end{matrix}\right]$}
  \put(285,110){$X_2$}
  \end{picture}
  \caption{Fusion Matrix Presentation}
  \vspace{\baselineskip}
  \end{subfigure}
  \begin{subfigure}{\textwidth}
  \centering
  \begin{picture}(300,140)
  \put(25,20){\tikzmark{b}\circle*{6}}
  \put(25,110){\tikzmark{a}\circle*{6}}
  \put(115,20){\tikzmark{c}\circle*{6}}
  \put(115,110){\tikzmark{d}\circle*{6}}
  \put(70,65){\tikzmark{e}\circle*{6}}
  \put(2,15){$X_2$}
  \put(2,105){$X_1$}
  \put(125,15){$X_3$}
  \put(125,105){$X_4$}
  \put(65,45){$X_5$}
  \put(185,20){\circle*{6}}
  \put(185,110){\circle*{6}}
  \put(275,20){\circle*{6}}
  \put(275,110){\circle*{6}}
  \put(230,65){\circle*{6}}
  \put(162,15){$X_2$}
  \put(162,105){$X_1$}
  \put(285,15){$X_3$}
  \put(285,105){$X_4$}
  \put(225,45){$X_5$}
  \begin{tikzpicture}[remember picture, overlay]
      \draw [{<[scale = 2]}-] ({pic cs:e}) +(-2pt,2pt) coordinate (ea)
      [out=135, in=-45] to ({pic cs:a});
  \draw [{<[scale = 2]}-] ({pic cs:e}) +(-2pt,-2pt) coordinate (eb)
  [out=-135, in=45] to ({pic cs:b});
  \draw [{<[scale = 2]}-] ({pic cs:e}) +(2pt,-2pt) coordinate (ec)
  [out=-45, in=135] to ({pic cs:c});
  \draw [{<[scale = 2]}-] ({pic cs:e}) +(2pt,2pt) coordinate (ed)
  [out=45, in=-135] to ({pic cs:d})
  ;
   \draw [{<[scale = 2]}-] ({pic cs:a}) +(-2pt,2pt) coordinate (aa)
      [out=135, in=45,looseness = 30] to ({pic cs:a});
  \draw [{<[scale = 2]}-] ({pic cs:b}) +(-2pt,-2pt) coordinate (bb)
  [out=-135, in=-45,looseness = 30] to ({pic cs:b});
  \draw [{<[scale = 2]}-] ({pic cs:c}) +(2pt,-2pt) coordinate (cc)
  [out=-45, in=-135,looseness = 30] to ({pic cs:c});
  \draw [{<[scale = 2]}-] ({pic cs:d}) +(2pt,2pt) coordinate (dd)
  [out=45, in=135,looseness = 30] to ({pic cs:d});
  \draw[CB_red,line width = 1.5pt] (185pt,120pt) arc (90:180:10pt) -- (175pt,20pt) arc
  (180:360:10pt) --  (195pt,90pt) arc (180:90:10pt) -- (275pt,100pt) arc
  (-90:90:10pt) -- (185pt,120pt);
  \draw[CB_purple,line width = 1.5pt] (183pt,12pt) arc (270:90:10pt) -- (253pt,32pt) arc (-90:0:10pt) --
  (263pt,112pt) arc (180:0:10pt) -- (283pt,112pt) -- (283pt,22pt) arc 
  (0:-90:10pt) -- (183pt,12pt);
  \draw[CB_lightblue,line width = 1.5pt] (185pt,118pt) arc (90:270:10pt) -- (255pt,98pt) arc (90:0:10pt) --
  (265pt,18pt) arc (-180:0:10pt) -- (285pt,108pt) arc 
  (0:90:10pt) -- (185pt,118pt);
  \draw[CB_green,line width = 1.5pt] (177pt,110pt) arc (180:00:10pt) -- (197pt,40pt) arc
  (-180:-90:10pt) -- (277pt,30pt) arc (90:-90:10pt) -- (187pt,10pt) arc 
  (-90:-180:10pt) -- (177pt,110pt);
    \end{tikzpicture}
  \end{picture}
  \caption{Graphical Presentation}
  \end{subfigure}
  \caption{Presentations of the Fusion Ring
  $\Rep\paren{\mbbZ_{3}^{2}\rtimes\mbbZ_{2}}$.}\label{F:rep}
  \end{figure} 

  It is worth noting that this digraph/hypergraph construction can easily be
  extended to non-multiplicity-free fusion rings by allowing for
  positive integer weights on the edges and hyperedges.  However, this
  is beyond of the scope of this work, and we will not consider it further.

\section{Fusion Rings Generated by Undirected Graphs}\label{S:undirected}
  In this section we will make use of the correspondence established above.
  Specifically we will consider natural sets of digraph, \textit{e.g.},
  undirected graphs, connected graphs, and triangle free graphs, and ask when
  such graphs can generate a fusion ring. Each graph result will be paired with
  the implication on the side of fusion categories. Many of the fusion category
  statements are verbose due to notions in graph theory not being naturally
  expressible in the language of fusion categories, \textit{e.g.}, neighborhood.
  Nonetheless, our results provide a new perspective on the classification and
  enumeration of fusion categories. The power of this approach is displayed in
  our main result, Theorem \ref{Theorem: Main Result}, where we are able to
  provide a complete characterization of fusion categories (up to Grothendieck
  equivalence) generated by triangle free undirected graphs.

  Throughout this section we will assume that the digraph $D$ is undirected,
  that is if $(i,j)$ is an edge in $D$, then so is $(j,i)$.  In order to
  emphasize this assumption we will use $G$ to denote the digraph and simply
  refer to a graph, hypergraph pair.  In terms of the fusion ring this condition
  is equivalent to:
  \begin{equation*}
    N_{ii}^j = N_{ji}^i = N_{ij}^i = N_{jj}^i = N_{ji}^j = N_{ij}^j\quad\text{for
    all $i,j\in[r-1]$.}
  \end{equation*}

  For notational brevity we introduce the following for a given pair
  $\paren{G,H}$:
  \begin{align*}
  e_{ij} &\equiv \textrm{indicator function for the presence of the edge
    $\{i,j\},$} \\
  w_{ij} &\equiv \textrm{number of $k$ such that $\{i,j,k\}$ is a
           hyperedge.}  \\
  \end{align*}
  Using the fact that the fusion matrices commute (in particular, that
  $\paren{N_iN_j}_{ij} = \paren{N_jN_i}_{ij}$) it is easy to see that
  the $w_{ij}$'s and the $e_{ij}$'s are related via 
  \begin{equation} \label{nec2}
    w_{ij} = 1+\sum_{k \in [r-1]}e_{ik}e_{jk}-2e_{ij}.
  \end{equation}
Combinatorially, this equation can be interpreted as saying that if $i$ and $j$ are adjacent, then the number of hyperedges is one less than the number of triangles containing $\{i,j\}$ in $G$ plus the number of loops on $i$ or $j$.  If $i$ and $j$ are not adjacent, then the number of hyperedges containing $\{i,j\}$ is one more than the number of common neighbors of $i$ and $j$ in $G$.

  \subsection{Empty Graphs}
    We first consider the situation where the graph $G$ has no edges, that is,
    $G$ consists of some number of isolated vertices. In this case, we
    can impose additionally structure on the associated hypergraph
    $H$, namely, that it is a $3$-uniform hypergraph where every pair
    of vertices is in precisely one hyperedge.  This is more commonly
    referred to as a Steiner triple system~\cite{lindner2017design}. 
    \begin{thm}\label{emptyGraph}
      Let $G$ be the empty graph on $n$ vertices and let $H$ be such that
      $(G,H)$ generates a fusion ring, then $n = 2^k-1$ for some $k$ and $H$ is
      a Steiner triple system of order $2^k-1$.  Furthermore,  the $H$
      that generates the fusion ring is unique up
      to isomorphism.
    \end{thm}
    \begin{proof}
    Since $G$ is an empty graph, for any $1 \leq i < j \leq n$ we have that $N_{ii}^j = N_{jj}^i = 0$.  As a consequence, for all such $i$ and $j$, $(N_iN_j)_{ij} = \sum_k N_{ii}^kN_{jk}^j = N_{ii}^0N_{j0}^0 = 1$.  Observing that by the commutativity we have that that $1 = (N_jN_i)_{ij} = \sum_k N_{ji}^kN_{ik}^j$ and that this counts the number of hyperedges containing $\{i,j\}$ plus the number of directed edges between $i$ and $j$, we conclude that there is a unique hyperedge containing $\{i,j\}$.  That is, $H = ([n],E)$ is a Steiner triple system. As Steiner triple systems are known to exist if and only if  $n 
      \equiv 1,2 \pmod{6}$, (see for instance \cite{lindner2017design}) this significantly restricts the possible self-dual, multiplicity-free fusion rings that can be generated by the empty graph.  Indeed, letting $H$ be an arbitrary Steiner triple system generates a set of valid fusion coefficients which respects rigidity and that the duality operation is an anti-involution.  However, not all such Steiner triple systems will resulting in a commuting family of fusion matrices.  

      In order to understand which Steiner triple systems result in a commuting family of fusion matrices, we define for each $i \in [n]$ the graph $G_i$ to be the graph whose adjacency matrix is $N_i$.  We note the fact that $N_{ia}^b = N_{ib}^a$ implies that this is $G_i$ is a graph rather than a digraph. 
      Observing
      that $(N_i)_{i0} = 1$ and that for all $i, j \in [n]$ there is a unique
      $k$ such that $(N_i)_{jk} = 1$, we conclude that each $N_i$ has precisely one non-zero entry in each row, and hence $G_i$ is a perfect
      matching. Finally, noting that for each (not necessarily distinct)
      $j, k \in \{0,\ldots,n\}$ there is precisely one $i \in \{0,\ldots,n\}$ such that $(N_i)_{j k} = 1$ and thus $\sum_{i = 0}^r N_i = J$ where
      $J$ is the $(n+1) \times (n+1)$ matrix of ones.  Now, as $N_0 = I$
      we have that $\sum_{i = 1}^r N_i = J - I$ which is the adjacency
      matrix of the complete graph on $n+1$ vertices.   
      Thus we have that the $G_i$ form a partition of the edges of
      $K_{n+1}$ into $n$ commuting perfect matchings. By 
      \cite[Theorem 1]{Akbari2006}, such a decomposition exists if and only if $n+1 = 2^k$
      for some $k \in \N$. Further, the pair $(G,H)$ generates a fusion ring if and only if $H$ is Steiner triple system defined from a commuting collection of perfect matchings.

      In order to prove uniqueness of $H$ (up to relabeling) we show that in any decomposition
      of $K_{2^k}$ into commuting perfect matchings there exists a
      collection of $k$ of these perfect matchings whose union forms a
      hypercube of dimension $k$.  Furthermore, for every perfect matching
      outside this set, the edge set is uniquely specified by any edge in
      the perfect matching.  

      By way of induction suppose that $M_1, \ldots, M_t$ is a collection of
      commuting perfect matchings on $[2^k]$ whose union consists of $2^{k-t}$
      disjoint $t$-dimensional hypercubes.  Now let $M$ be a perfect matching that
      commutes with and is disjoint from all of the $M_i$.  From
      \cite{Akbari2006}, this implies that $M \cup M_i$ is a collection of
      disjoint 4-cycles for all $i$.  Let $S$ be the vertex set of one of the
      $t$-dimensional hypercubes in $\cup_i M_i$.  We claim that either $M[S]$ has
      zero edges or $2^{t-1}$ edges.  To this end suppose $e$ is an edge in $M[S]$
      and without loss of generality label the vertices of $S$ with length $t$
      binary strings such that one endpoint of $e$ is the binary $0$-string string
      and for all $i$ the endpoints of the edges in $M_i$ differ only in the
      $i^{\textrm{th}}$ component.  Now we note that since $M$ and $M_i$ commute,
      if $e$ is an edge in $M[S]$, then the edge with $i^{\textrm{th}}$ component
      of both endpoints of $e$ changed is also in $M[S]$.  We will abuse notation
      and denote this edge by $e + e_i$.  By a similar argument we have that for
      $j \neq i$, $(e+e_i) + e_j$ is also in $M[S]$.  Noting that this operation
      is commutative, the action of each matching $M_i$ is of order two, and that
      there is a collection of actions which reverse the labels of the endpoints
      of $e$, we have that $M[S]$ contains at least $2^{t-1}$ edges.  In fact, it
      contains exactly this many as $M$ is a perfect matching and $\size{S} =
      2^t$.

      Thus, for every hypercube in $\cup M_i$, there are at most $2^t - t -
      1$ commuting perfect matchings on $[2^k]$ containing an edge internal
      to the hypercube.   As there are $2^{k-t}$ such hypercubes, there are
      at most $2^{k-t}\paren{2^t - t - 1} = 2^k - (t+1)2^{k-t}$ commuting
      hypercubes which have an edge in any of the hypercubes in $\cup M_i$.
      Since there are at $2^k - t - 1$ unused matchings, there is some
      matching $M_{t+1}$ which does not contain an edge internal to any
      hypercube in $\cup M_i$.  By a similar argument as above, $M_{t+1}
      \cup \bigcup_i M_i$ forms $2^{k-t-1}$ hypercubes of dimension $t+1$
      on $[2^k]$.  Repeating this process yields $M_1,\ldots, M_k$, a
      collection of $k$ commuting perfect matchings on $[2^k]$ whose union
      is a hypercube.  The above argument also gives that every matching
      outside $\{M_1,\ldots,M_t\}$ is completely determined by a single
      edge, as desired.
    \end{proof}

  \begin{Acor}
    If $\mcC$ is a multiplicity-free, self-dual, braided fusion category of rank $n+1$ such that every element is invertible, then $n=2^{k}-1$ for some $k\in\mbbZ$ and $\mcC$ is Grothendieck equivalent
    to $\Rep\paren{\mbbZ_{2}^{k}}$.
  \end{Acor}

  \subsection{Simple Components}
    In order to further our understanding of which undirected graphs $G$
    can generate a fusion ring we will next consider the non-trivial
    components $G$ which do not contain a self-loop.

    \begin{lem}
      \label{neighbors}
      Let $G$ be an undirected graph that generates a fusion ring and let $C$ be
      a nontrivial component 
      of $G$ which contains no self-loops.  For every vertex $v \in C$, the minimum degree of $C[N(v)]$
      is at least 2.
    \end{lem}
    \begin{proof}
      Suppose not. Then there exists a vertex $v$ which has a neighbor $u$ with
      degree at most 1 in $C[N(v)]$. In fact, $u$ must have a exactly one
      neighbor in $C[N(v)]$ since otherwise, by Equation \ref{nec2}, $\{u,v\}$
      would participate in $-1$ hyperedges. Let $x$ be the unique vertex that is
      adjacent to both $u$ and $v$. Then noting that $w_{uv} = 0$ by Equation
      \ref{nec2}, we have that
      \begin{align*}
        (N_vN_u)_{xv} &= \sum_k N_{vx}^kN_{uk}^v = N_{vx}^{u}e_{uv} + e_{vx}e_{vu} = e_{vx}e_{vu} = 1, \textrm{ while } \\
        (N_vN_u)_{vx} &=\sum_{\ell}N_{vv}^{\ell}N_{u\ell}^x \geq e_{vu}e_{ux}+e_{vx}e_{xu} = 2
      \end{align*}
      a contradiction.
    \end{proof}

    \begin{Acor}
      If $\mcC$ is a multiplicity-free, self-dual, braided fusion category such
      that $\dim \Hom_{\mcC}\paren{X\otimes X,X} = 0$ for all $X \in \Irrstar(\mcC)$ and 
     $\dim \Hom_{\mcC}\paren{X\otimes X,Y} = \dim \Hom_{\mcC}\paren{Y\otimes Y,X}$ for all distinct objects $X,Y \in \Irrstar(\mcC)$, then     
      then for all non-invertible $X\in \Irrstar\paren{\mcC}$, $X\otimes X$ contains
      at least three non-isomorphic subobjects from $\Irrstar\paren{\mcC}$.
    \end{Acor}

    It is an easy consequence of Lemma \ref{neighbors} that every vertex and
    edge in such a loopless component participates in at least 2 triangles.
    Consequently the minimum degree of such a component is at least 3.  In fact,
    with care we can show that the minimum degree of such a component is at
    least 4.

    \begin{lem}
      \label{L:mindeg}  Suppose that $G$ is an undirected graph which
      generates a fusion ring and let $C$ be a non-trivial
      component of $G$ with no self-loops, then the minimum degree of $C$ is at least 4.
    \end{lem}
    \begin{proof}
      Let $H$ be a hypergraph such that $(G,H)$ generates a fusion ring and
      let $C$ be a non-trivial simple component of $G$ with minimum degree
      at most $3$.  As a consequence of Lemma \ref{neighbors}, the component $C$
      contains a vertex $\ell$ of degree exactly 3. Furthermore, $C[N(\ell) \cup
      \{\ell\}]$ is a complete graph on 4 vertices, $\{i,j,k,\ell\}.$
      By Equation \ref{nec2} we have that $w_{\ell i} = w_{\ell j} =
      w_{\ell k} = 1$.  Suppose now that $\{i,j,\ell\}$ is the
      unique hyperedge involving $i$ and $\ell$, and note that
      \begin{equation*}
        (N_{\ell}N_i)_{i j} = \sum_u N_{\ell i}^u N_{i u}^j = e_{\ell i} N_{i \ell}^j + e_{i \ell} e_{i j} + N_{\ell i}^je_{j i} = 3,
      \end{equation*}
      while
      \begin{equation*}
        (N_{\ell}N_i)_{j i} = \sum_v N_{\ell j}^v N_{i v}^i = e_{j
          \ell}e_{i j} + e_{\ell j}e_{i \ell} + N_{\ell j}^i e_{ii} = 2,
      \end{equation*}
      which is a contradiction to the commutativity of $N_{\ell}$ and $N_i$.  As $i$ and
      $j$ were arbitrary, we have that $\{i, k, \ell\}$ and $\{j, k, \ell\}$
      are also not hyperedges in $H$.

      Thus, there is some vertex $s \notin \{i,j,k,\ell\}$ such that
      $\{s,i,\ell\}$ is a hyperedge. Now note that
      \begin{equation*}
        (N_{\ell}N_j)_{\ell i} = \sum_u N_{\ell \ell}^u N_{j u}^i = e_{\ell i} e_{i j} + e_{\ell j} e_{j i} + e_{\ell k} N_{j k}^i = 2 + N_{j k}^i
      \end{equation*}
      and 
      \begin{equation*}
        (N_{\ell}N_j)_{i \ell} = \sum_v N_{\ell i}^v N_{j v}^{\ell} = e_{\ell i} e_{\ell j}  + e_{i \ell} N_{j i}^{\ell} + N_{\ell i}^s N_{j s}^{\ell} = 1 + N_{j \ell}^s.
      \end{equation*}
      Thus, by the symmetry of $N_{\ell}N_i$ we have that $\{i,j,k\} \notin
      E(H)$ and $\{s, j, \ell\} \in E(H)$.  A similar argument shows that $\{s,
      k,\ell\} \in E(H)$.  Figure \ref{F:first_step} illustrates the required
      edges and hyperedges thus far.

      \begin{figure}
        \hfill
        \subfloat[Edges in $G$]{
          \begin{tikzpicture}
            \draw[fill = black] (0,0) circle (.1cm);
            \node at (0,-.35) {$\ell$};
            \draw[fill = black] (-1.5,1.5) circle (.1cm);
            \node at (-1.5,1.15) {$i$};
            \draw[fill = black] (0,3) circle (.1cm); 
            \node at (-.25,3) {$j$};
            \draw[fill = black] (1.5,1.5) circle (.1cm);
            \node at (1.5,1.15) {$k$};
            \draw[fill = black] (0,4.5) circle (.1cm);
            \node at (-.25,4.5) {$s$};
            \draw (0,0) -- (1.5,1.5) -- (0,3) -- (-1.5,1.5) -- (0,0);
            \draw (1.5,1.5) -- (-1.5,1.5);
            \draw (0,0) -- (0,3);
          \end{tikzpicture}
        }
        \hfill\hfill
        \subfloat[Hyperedges in $H$]{
          \begin{tikzpicture}
            \draw[fill = black] (0,0) circle (.1cm);
            \node at (0,-.35) {$\ell$};
            \draw[fill = black] (-1.5,1.5) circle (.1cm);
            \node at (-1.5,1.15) {$i$};
            \draw[fill = black] (0,3) circle (.1cm); 
            \node at (-.25,3) {$j$};
            \draw[fill = black] (1.5,1.5) circle (.1cm);
            \node at (1.5,1.15) {$k$};
            \draw[fill = black] (0,4.5) circle (.1cm);
            \node at (-.25,4.5) {$s$};
            \path plot[smooth cycle, tension =0.2] coordinates{(-.4,-.4) (-.4,4.6) (.4,4.6) (.4,-.4)}; 
            \path plot[smooth cycle, tension= 0.2] coordinates{(.57,0) (0,-.57) (-2.06,1.44) (-.18,5.04) (.54,4.78) (-1.05,1.62)}; 
            \path plot[smooth cycle, tension= 0.2] coordinates{(-.57,0) (0,-.57) (2.06,1.44) (.18,5.04) (-.54,4.78) (1.05,1.62)};
            \path plot[smooth cycle, tension = 0.2] coordinates{(-.53,-.21) (.21, -.53) (2,3.64) (-.11,5.05) (-.55,4.39) ( 1,3.36 )};
            \draw plot[smooth cycle, tension =0.2] coordinates{(-.4,-.4) (-.4,4.6) (.4,4.6) (.4,-.4)}; 
            \draw plot[smooth cycle, tension= 0.2] coordinates{(.57,0) (0,-.57) (-2.06,1.44) (-.18,5.04) (.54,4.78) (-1.05,1.62)}; 
            \draw plot[smooth cycle, tension= 0.2] coordinates{(-.57,0) (0,-.57) (2.06,1.44) (.18,5.04) (-.54,4.78) (1.05,1.62)};
            \path plot[smooth cycle, tension = 0.2] coordinates{(-.53,-.21) (.21, -.53) (2,3.64) (-.11,5.05) (-.55,4.39) ( 1,3.36 )};
          \end{tikzpicture}
        }
        \hfill\hfill
        \caption{Initial edges and hyperedges required by vertex of degree 3.}\label{F:first_step}
      \end{figure}

      As $s \not\sim \ell$ and $w_{\ell s} \geq 3$, we have by Equation
      \ref{nec2} that $\ell$ and $s$ have at least two common
      neighbors in $G$.  Without
      loss of generality suppose that $i$ and $j$ are common neighbors of $\ell$
      and $s$.  At this point, we have that either $s \not\sim k$ and $w_{\ell
      s} = 3$ or $s \sim k$ and there exists a vertex $t \not\in \{i,j,k\}$ such that $\{\ell, s,
      t\} \in E(H)$.  Now considering 
      \begin{equation*}
        (N_{\ell}N_k)_{k s} = \sum_u N_{\ell k}^u N_{k u}^s = e_{\ell k} N_{k \ell}^s + e_{k \ell} e_{k s} + N_{\ell k}^s e_{s k} = 1 + 2e_{k s},
      \end{equation*}
      and
      \begin{equation*}
        (N_{\ell}N_k)_{sk} = \sum_v N_{\ell s}^vN_{k v}^k = N_{\ell
          s}^{i}e_{k i} + N_{\ell s}^{j} e_{k j} + N_{\ell s}^te_{k t}
        + N_{\ell s}^k e_{kk} = 2 + N_{\ell s}^te_{k t},
      \end{equation*}
      we have that $1 + e_{k s} = 2 + N_{\ell s}^te_{k,t}$.  Thus, $s \sim k$, $\{\ell, s, t\} \in E(H)$, and $k \sim t$.  But
      as we now know that $k \sim s$, this implies that $i$ and $j$
      were chosen arbitrarily.  Repeating the
      previous argument for $\{i,k\}$ and $\{j,k\}$ and using that
      $w_{\ell s} = 4$, we get that $i \sim t$ and $j \sim t$.  As a
      result we have the edges and hyperedges depicted in Figure
      \ref{F:second_step}.

      \begin{figure}
\centering
        \hfill
        \subfloat[Edges in $G$]{
          \begin{tikzpicture}
            \draw[fill = black] (0,0) circle (.1cm);
            \node at (0,-.35) {$\ell$};
            \draw[fill = black] (-1.5,1.5) circle (.1cm);
            \node at (-1.5,1.15) {$i$};
            \draw[fill = black] (0,3) circle (.1cm); 
            \node at (-.25,3) {$j$};
            \draw[fill = black] (1.5,1.5) circle (.1cm);
            \node at (1.5,1.15) {$k$};
            \draw[fill = black] (0,4.5) circle (.1cm);
            \node at (-.25,4.5) {$s$};
            \draw[fill = black] (1.5,3.5) circle (.1cm); 
            \node at (1.5,3.75) {$t$};
            \draw (0,0) -- (1.5,1.5) -- (0,3) -- (-1.5,1.5) -- (0,0);
            \draw (1.5,1.5) -- (-1.5,1.5);
            \draw (0,0) -- (0,3);
            \draw (0,4.5) -- (-1.5,1.5) -- (1.5,3.5) -- (1.5,1.5) -- (0,4.5) -- (0,3) -- (1.5,3.5);

            \path plot[smooth cycle, tension =0.2] coordinates{(-.4,-.4) (-.4,4.6) (.4,4.6) (.4,-.4)}; 
            \path plot[smooth cycle, tension= 0.2] coordinates{(.57,0) (0,-.57) (-2.06,1.44) (-.18,5.04) (.54,4.78) (-1.05,1.62)}; 
            \path plot[smooth cycle, tension= 0.2] coordinates{(-.57,0) (0,-.57) (2.06,1.44) (.18,5.04) (-.54,4.78) (1.05,1.62)};
            \path plot[smooth cycle, tension = 0.2] coordinates{(-.53,-.21) (.21, -.53) (2,3.64) (-.11,5.05) (-.55,4.39) ( 1,3.36 )};
          \end{tikzpicture}
        }
        \hfill\hfill
        \subfloat[Hyperedges in $H$]{
          \begin{tikzpicture}
            \draw[fill = black] (0,0) circle (.1cm);
            \node at (0,-.35) {$\ell$};
            \draw[fill = black] (-1.5,1.5) circle (.1cm);
            \node at (-1.5,1.15) {$i$};
            \draw[fill = black] (0,3) circle (.1cm); 
            \node at (-.25,3) {$j$};
            \draw[fill = black] (1.5,1.5) circle (.1cm);
            \node at (1.5,1.15) {$k$};
            \draw[fill = black] (0,4.5) circle (.1cm);
            \node at (-.25,4.5) {$s$};
            \draw[fill = black] (1.5,3.5) circle (.1cm); 
            \node at (1.5,3.75) {$t$};
            \draw[CB_red,line width = 1.5pt] plot[smooth cycle, tension = 0.2] coordinates{(-.4,-.4) (-.4,4.6) (.4,4.6) (.4,-.4)}; 
            \draw[CB_green, line width = 1.5pt] plot[smooth cycle, tension = 0.2] coordinates{(.57,0) (0,-.57) (-2.06,1.44) (-.18,5.04) (.54,4.78) (-1.05,1.62)}; 
            \draw[CB_purple, line width = 1.5pt] plot[smooth cycle, tension = 0.2] coordinates{(-.57,0) (0,-.57) (2.06,1.44) (.18,5.04) (-.54,4.78) (1.05,1.62)};
            \draw[CB_lightblue, line width = 1.5pt] plot[smooth cycle, tension = 0.2] coordinates{(-.53,-.21) (.21, -.53) (2,3.64) (-.11,5.05) (-.55,4.39) ( 1,3.36 )};
          \end{tikzpicture}
        }
        \hfill\hfill
        \caption{Edges and hyperedges required by neighbors of vertex of degree 3.}\label{F:second_step}
      \end{figure}

      Now consider $i$ and $j$. By Equation \ref{nec2}, $w_{ij} \geq 3$.
      Further, from the above, we know that $\{i,j,k\}$ and $\{i,j,\ell\}$ are
      not hyperedges. Thus, there exists a vertex $a \notin \{\ell, k, s, t\}$
      in the graph such that $\{i,j,a\}$ is a hyperedge. Now note that
      \begin{equation*}
        (N_{\ell}N_i)_{\ell a} = \sum_u N_{\ell\ell}^uN_{i u}^a \geq e_{\ell j} N_{i j}^a \geq 1,
      \end{equation*}
      and 
      \begin{equation*}
        (N_{\ell}N_i)_{a \ell} = \sum_v N_{\ell a}^v N_{i v}^{\ell} = e_{\ell a} e_{\ell i} + N_{\ell a}^i e_{i v} + N_{\ell a}^s N_{i s}^{\ell} = 0,
      \end{equation*}
      contradicting the symmetry of $N_{\ell}N_i$.  Thus, we may conclude that
      the minimum degree in $C$ is at least 4.
    \end{proof}

    \begin{Acor}
      If $\mcC$ is a self-dual, multiplicity-free, braided fusion category such that 
     $\dim \Hom_{\mcC}\paren{X\otimes X,X} = 0$ for all $X \in \Irrstar(\mcC)$ and 
     $\dim \Hom_{\mcC}\paren{X\otimes X,Y} = \dim \Hom_{\mcC}\paren{Y\otimes Y,X}$ for all distinct objects $X,Y \in \Irrstar(\mcC)$,
       then all for all non-invertible
      $Y\in\Irrstar\paren{\mcC}$ the product $Y\otimes Y$ contains at least 4 non-isomorphic 
      subobjects from $\Irrstar\paren{\mcC}$.
    \end{Acor}

    \begin{rem}
      In \cite{Fijavz2010}, they show that an undirected graph with minimum
      degree at least 4 contains either $K_5$ or $K_{2,2,2}$ as a minor.
      Further, by Kuratowski's Theorem, a graph with a $K_5$ minor is nonplanar.
      Thus a graph coming from a fusion ring which contains an undirected,
      simple, nontrivial component is either nonplanar or contains a $K_{2,2,2}$
      minor. 
    \end{rem} 

  \subsection{Non-simple Components}
    In order to understand the behavior of the non-simple components of an
    undirected graph generating a fusion ring it will be helfpul to
    rephrase implications of Equation \ref{nec2} in terms of the local
    loop structure of graph.

    \begin{lem} 
      \label{forceloop}
      Let $G$ be an undirected graph that generates a fusion ring. 
      \begin{enumerate}
        \item If $i \not\sim j$ in $G$, then $w_{ij} = 1+|N(i)\cap N(j)|.$
        \item If $i \sim j$ in $G$ and $(N(i) \cap N(j))\setminus \{i,j\} = \emptyset$, then
        \begin{itemize}
          \item $w_{ij} = 0$ and there is a loop in $G$ on precisely one of $i$ and $j$, or
          \item $w_{ij} = 1$ and there is a loop on both vertices in $G$.
        \end{itemize}
        \item If $i \sim j$ in $G$ and both vertices are loopless, then they must share a common neighbor. \label{E:loopless}
      \end{enumerate} 
    \end{lem}

    Our first observation shows that if an undirected graph $G$ generates
    a fusion ring then it has at most one component containing loops.  
    \begin{lem} 
      \label{loopdist}
      Let $G$ be an undirected graph that generates a fusion ring, then for any
      two looped vertices $i$ and $j$ in $G$, $d(i,j) \leq 2$.
    \end{lem}
    \begin{proof}
      By way of contradiction, suppose that $i$ and $j$ are looped vertices in $G$ and $d(i,j) \geq 3$.   By Lemma
      \ref{forceloop},  $w_{ij} = 1$. Now let $\{i,j,k\}$ be the unique
      hyperedge containing $i$ and $j$ and  note that
      \begin{equation*}
        (N_iN_j)_{ik} = \sum_{u} N_{ii}^u N_{j u}^k \geq e_{ii}N_{j i}^k = 1,
      \end{equation*}
      and
      \begin{equation*}
        (N_iN_j)_{ki} = \sum_v N_{ik}^vN_{jv}^i = e_{ik}e_{ij} + e_{ki}N_{jk}^i + N_{ik}^j e_{ji} = e_{ki}
      \end{equation*}
      Thus, by the symmetry of $N_iN_j$ we must have the edge $\{i,k\}$ in $G$.
      Further, by comparing $(N_iN_j)_{jk}$ and $(N_iN_j)_{kj}$, $\{j,k\}$ must
      also be present, contradicting that $d(i,j) \geq 3$.  Thus either $i \sim j$ or $d(i,j) = 2$, as desired. 
    \end{proof}

    \begin{Acor}
      If $\mcC$ is a multiplcity-free, self-dual, braided fusion category such that 
     $\dim \Hom_{\mcC}\paren{X\otimes X,Y} = \dim \Hom_{\mcC}\paren{Y\otimes Y,X}$ for all distinct objects $X,Y \in \Irrstar\paren{\mcC}$, then
      for distinct subobjects $X, Y\in\Irrstar\paren{\mcC}$ such that $\dim \Hom_{\mcC}\paren{X\otimes X,X} = \dim \Hom_{\mcC}\paren{Y\otimes Y,Y} = 1$ either $Y$ is a subobject of $X\otimes
      X$, or $X\otimes X$ and $Y\otimes Y$ have a common subobject in $\Irrstar\paren{\mcC}$.
    \end{Acor}

    As a consequence of this result, there is only one non-simple
    component in graph $G$ which generates a fusion ring, and furthermore
    that component has diameter at most 2.  It might seem that the diameter
    2 constraint would significantly limit the number of simple graphs that need
    be considered as the base for the non-simple component in $G$, however
    a classic result of Erd\H{o}s and Reny\'{i} gives that a vanishingly
    small fraction of graphs have diameter greater than 2, see \cite{Bollobas:RandomGraphs}.

    \begin{lem}
      \label{loops} 
      Let $G$ be an undirected graph that generates a fusion ring.  If $G$
      contains an  induced path, $i \sim j \sim k$ where $\{i,j\}$
      is not in a triangle, then $i$ and $j$ both have self loops.
    \end{lem}
    \begin{proof}
      By Lemma \ref{forceloop}, it suffices to show that there is a
      hyperedge containing $i$ and $j$.  To that end we note that
      \begin{equation*}
        1 = N_{ii}^jN_{kj}^j 
          \leq \sum_u N_{ii}^uN_{ku}^j 
          = (N_iN_k)_{ij} 
          = (N_iN_k)_{ji} 
          = \sum_v N_{ij}^vN_{kv}^i  
          \leq  e_{ji}N_{kj}^i + \sum_{v \neq i,j} N_{ij}^v 
          = N_{kj}^i + w_{ij}.
      \end{equation*}
      Thus there is a hyperedge containing both $i$ and $j$, as desired.
    \end{proof}

    \begin{Acor}
      Let $\mcC$ be a multiplicity-free, self-dual, braided fusion category such that  
     $\dim \Hom_{\mcC}\paren{X\otimes X,Y} = \dim \Hom_{\mcC}\paren{Y\otimes Y,X}$ for all distinct objects $X,Y \in \Irrstar(\mcC)$. If
      $X,Y,Z$ are distinct elements of $\Irrstar\paren{\mcC}$ such that $X$ and $Z$ are subobjects of $Y \otimes Y$ and $Z$ is not a subobject of $X \otimes X$, then either,
      $X\otimes X$ and $Y\otimes Y$ have a common subobject other
      than $\one$, $X$, and $Y$ or
      $\dim\Hom_{\mcC}\paren{X\otimes X,X}=\dim\Hom_{\mcC}\paren{Y\otimes
      Y,Y}=1$.
    \end{Acor}

    Note that this implies that any triangle-free component with at least 3
    vertices of a graph generating a fusion ring has a self-loop at every
    vertex.

\subsection{Fusion Rings Generated by Triangle-Free Graphs}
We now turn to the principle result of this section, a complete
characterization of the undirected, triangle-free graphs which
generate a fusion ring.  

 \begin{thm}\label{T:Tfree}
      The undirected, triangle-free graphs which generate fusion rings are: 
      \begin{enumerate}
        \item a single vertex with a self loop,
        \item a single edge with one self loop,
        \item a single edge with self loops on each vertex and an isolated
        loopless vertex, and
        \item a collection of $2^k-1$ isolated loopless vertices. 
      \end{enumerate}
    \end{thm} 

    \begin{proof}
  Let $G$ be a triangle-free, undirected graph which generates a
  fusion ring.  In the case that $G$ is the empty graph, by Theorem \ref{emptyGraph}, $G$ has order $2^k-1$ and has a unique hypergraph $H$
      associated with which it generates a fusion ring. Thus we may
      assume that $G$ contains at least one edge and let $T$ be the maximal collection of nontrivial triangle-free components.
      By Lemma \ref{forceloop} every edge in $T$ has a self-loop on at least one
      of the vertices.  But then by Lemma \ref{loopdist} the diameter of $T$ is
      at most $4$ and in particular $T$ is a single component.  In particular, $G$ consists of $T$ and some number of trivial (that is, isolated without a self-loop) vertices.

Suppose now that $T$ contains at least 3 vertices.  Since $T$ is connected and contains at least $3$ vertices, for any edge $\{i,j\}$ in $T$, there exists some vertex $k$ such that either $i \sim k$ or $j \sim k$.  Since $T$ is triangle-free, by Lemma \ref{loops} both $i$ and $j$ have a self-loop.  As $\{i,j\}$ is an arbitrary edge, this implies that all vertices in $T$ have a self-loop and by Lemma \ref{loopdist} $T$ has diameter at most 2.  In fact, since $T$ is triangle-free and has at least three vertices, this implies that the diameter of $T$ is exactly 2.  Furthermore, by Lemma \ref{forceloop} for
      every edge $\{i,j\} \in E(T)$, $w_{ij} = 1$.  In particular,
      any hypergraph which, with $G$, generates a fusion ring is
      non-empty.  Let $H$ be one such hypergraph and let $f = \{i,j,k\}$ be an arbitrary edge in $H$.  Since $T$ is
      triangle-free, there are at most two edges of $T$ in $G[f]$.  Suppose then
      that there are precisely two edges from $T$ in $f$, namely $i \sim j$ and
      $j \sim k$.  Since $i,j,k$ are all in $T$, we have that
      $e_{ii}=e_{jj}=e_{kk}=1$ and, by Lemma \ref{forceloop}, $w_{ik} = w_{kj} =
      1$.  Furthermore, the hyperedge indicated by $w_{ij}$ and $w_{jk}$ is $f$.
      But now we have 
      \begin{equation*}
        (N_iN_j)_{ik} =\sum_{u}e_{iu}N_{ju}^k = e_{ii}N_{ji}^k+
        e_{ik}e_{kj}  + e_{ij}e_{jk} =
        2.
      \end{equation*}
      and
      \begin{equation*}
        (N_iN_j)_{ki} = \sum_v N_{ik}^vN_{jv}^i = e_{ki}N_{jk}^i  +
        e_{ik}e_{ij} + N_{ik}^j e_{ji} = 1,
      \end{equation*}
      which contradicts the symmetry of $N_iN_j$.  Thus, for any hyperedge $f$,
      $G[f]$ must contain at most one edge from $T$.

      We note that by Lemma \ref{forceloop} and the non-trivial nature of $T$,
      there exists some hyperedge $f$ such that $G[f]$ contains one
      edge of $T$.
      More specifically, suppose that $i$ and $j$ are adjacent vertices in $T$,
      then $w_{ij} = 1$.  Let $k$ be the third vertex in the unique hyperedge
      containing $i$ and $j$. Suppose $k$ lies in $T$.  Since $T$ is triangle
      free and has diameter 2, there must exist a vertex $x$ such that
      $i,j,x,k$ forms an induced a path of length 4 in $T$.  Now since $w_{xj} =
      1$ by Lemma \ref{forceloop}, we have that
      \begin{equation*}
        (N_xN_i)_{xj} = \sum_{u} e_{xu}N_{iu}^j = e_{xj}e_{ji}+e_{xk}N_{ik}^j + e_{xi}e_{ij} = 2,
      \end{equation*}
      and 
      \begin{equation*}
        (N_xN_i)_{jx} = \sum_{v} N_{xj}^vN_{iv}^x \leq  e_{xj}e_{xi} + e_{jx}N_{ij}^x + \sum_{s \neq j,x} N_{xj}^sN_{is}^x \leq 1,
      \end{equation*}
      again contradicting the symmetry of $N_xN_i$.  Thus for any hyperedge
      containing an edge in $T$, the third vertex in the hyperedge must not
      belong to $T$.

      Now suppose that $i$ and $j$ are two non-adjacent vertices in $T$.  Since
      $T$ is triangle free and has diameter 2, there is a vertex $k$ such that
      $i \sim k$ and $j \sim k$. But then we have that,
      \begin{equation*}
        1 = e_{ik}e_{jk} \leq \sum_u N_{ii}^uN_{ju}^k = (N_iN_j)_{ik} = (N_iN_j)_{ki} = \sum_v N_{ik}^vN_{jv}^i = \sum_{v \neq 0,i,j,k} N_{ik}^vN_{jv}^i.
      \end{equation*}
      Thus there is some vertex $\ell$ such that both $\{i,k,\ell\}$ and
      $\{i,j,\ell\}$ are hyperedges in $H$.  But since $i \sim k$, we know that
      $\ell$ is not in $T$, and in particular, $i$ and $\ell$ have no common
      neighbors.  Thus, by Lemma \ref{forceloop}, $w_{i\ell} = 1$, a
      contradiction.  Hence there are no non-adjacent vertices in $T$, and so
      $T$ has at most two vertices.
      That is, $T$ is a looped vertex or an edge
       with either one or two self-loops. 

      First suppose that $T$ is a looped vertex $x$. By Lemma \ref{forceloop}
      any trivial vertices must occur in exactly one hyperedge with $x$. 
      This would require two isolated vertices, $i$ and $j$, being contained in a hyperedge
      with a looped vertex.  However, this implies that $1 = (N_iN_j)_{ix} = (N_iN_j)_{xi} = 0$, a contradiction.  Thus, if $T$ is a single looped vertex, then $T$ is the entire graph $G$. Next,
      suppose that $T$ is an edge $\{x,y\}$ with at least one loop. Without loss
      of generality, first suppose that only $x$ is looped. Since $w_{xy} = 0$,
      additional trivial vertices must occur, in pairs, in exactly one hyperedge
      with $x$. Again we have that $T$ is the entire graph $G$.  Finally suppose
      that both $x$ and $y$ are looped.  By Lemma \ref{forceloop}, $w_{xy} = 1$
      and thus at least one trivial vertex must be present. However, if more
      than one trivial vertex is present, a hyperedge of the form $\{x,i,j\}$
      must also be present, where $i$ and $j$ are trivial vertices. Thus, such
      an edge must occur with exactly one trivial vertices. 
    \end{proof}

    From the fusion rules corresponding to the graphs in Theorem
    \ref{T:Tfree} it is easy verify the following theorem.

    \begin{thm}
      \label{Theorem: Main Result}
      The multiplicity-free, self-dual braided fusion categories generated by an
      undirected, triangle-free graph $G$ are the Grothendieck equivalent to one
      of the following:
      \begin{enumerate}
        \item $\Fib$
        \item $PSU(3)_2$,
        \item $PSU(2)_6$, and 
        \item $Rep(G)$ where $G$ is an elementary abelian 2-group of order $2^k$. 
      \end{enumerate}
    \end{thm}

\section{Low Rank Fusion Rings}\label{S:lowrank}
One of the advantages of the digraph representation of fusion rings is
that it allows for the use of already existing tools to facilitate
the complete enumeration.  For example, in this section we use the
utilities distributed McKay and Piperno's graph isomorphism tool
\texttt{nauty}/\texttt{Traces}~\cite{nauty} to generate all
non-isomorphic digraphs on a given number of vertices. Each digraph
then yields restrictions on the associated hypergraph through the
relationship between $e_{ij}$ and $w_{ij}$.  From simple combinatorial
considerations on the hypergraph, a significant fraction of digraphs can be
shown not to generate a fusion ring.  For the remaining digraphs, a direct
search through the associated hypergraphs will yield all digraph-hypergraph pairs which generate a fusion ring.  A complete table of
those pairs through rank 8 is provided below.  In addition to the fusion rings, we provide the Grothendieck class for
the fusion rings, if known.  To facilitate this procedure we make the
following observation:
\begin{obs}
If two digraphs, $D_1$ and $D_2$, generate fusion rings $F_1$ and
$F_2$, there is a digraph $D$ that is a super graph of $D_1 \times
D_2$ which generates $F_1 \boxtimes F_2$.  More concretely let $D_1'$ be
formed from $D_1$ by adding the vertex $0$ which has a self-loop and
directed edges from every vertex in $D_1$ to $0$.  Then $D = D_1'
\times D_2' - \{(0,0)\}$ and $D$ generates the fusion ring $F_1
\boxtimes F_2$.  
\end{obs}

\begin{rem}
It is perhaps interesting to note that in the tables below, the
tuple \[\left(\sum_i N_{ii}^i,  \sum_i \sum_{j \neq i} N_{ii}^j,
  \sum_{\size{\{i,j,k\}}=3} N_{ij}^k\right)\] uniquely determines the fusion
ring.  Furthermore, in examining the rank 5 fusion rings we see that
the more algebraically natural tuples $\left(\sum_i \Tr{N_i},
  \sum_{\size{\{i,j,k\}}=3} N_{ij}^k\right)$ or $\left( \sum_i
  \Tr{N_i}, \sum_{i,j,k} N_{i,j}^k\right)$ fail to distinguish between
all fusion rings.  We are not so bold as to conjecture that this holds
in general, but it may be a fruitful direction of future research.  
\end{rem}

\subsection*{Acknowledgments}
The last author would like to thank Tobias Hagge for his help in formulating the algebraic corollaries as well as Gert Vercleyen and S\'ebastien Palcoux for their helpful pointers to more recent results on the enumeration of fusion rings and categories.

\begin{longtable}{|c| M | M | c|}
\caption{Rank 2 Fusion Rings\label{T:Rank2}}\\
\hline
Loops & Arcs & Hyperedges & Grothendick Class \\ \hline
 \endfirsthead
\caption{Rank 3 Fusion Rings (continued)}\\ \hline
Loops & Arcs & Hyperedges & Grothendick Class \\ \hline
\endhead
& & & $\Sem$\cite{RSW}\\\hline
$1$ & && $\Fib$\cite{RSW}\\\hline
\end{longtable}

\begin{longtable}{|c| M | M | c|}
\caption{Rank 3 Fusion Rings\label{T:Rank3}}\\
\hline
Loops & Arcs & Hyperedges & Grothendick Class \\ \hline
 \endfirsthead
\caption{Rank 3 Fusion Rings (continued)}\\ \hline
Loops & Arcs & Hyperedges & Grothendick Class \\ \hline
\endhead
&$(1, 2)$ && $\Ising$\cite{O3}\\\hline
$1$ &$(1, 2)$ && $\Rep\paren{\mfS_{3}}$\cite{O3}\\\hline
$1$ &$(1, 2)$, $(2, 1)$ && $\PSU(3)_2$ \\ \hline
\end{longtable}

\begin{longtable}{|c| M | M | c|}
\caption{Rank 4 Fusion Rings\label{T:Rank4}}\\
\hline
Loops & Arcs & Hyperedges & Grothendick Class \\ \hline
 \endfirsthead
\caption{Rank 4 Fusion Rings (continued)}\\ \hline
Loops & Arcs & Hyperedges & Grothendick Class \\ \hline
\endhead
&&$(1, 2, 3)$ & $\Sem^{2}$\cite{RSW}\\\hline
&$(1, 2)$, $(1, 3)$, $(2, 1)$, $(2, 3)$ && $\PSO(5)_6$\cite{B1}\\ \hline
$2$ &$(1, 2)$ &$(1, 2, 3)$ & $\Sem\boxtimes\Fib$\cite{RSW}\\\hline
$1$, $2$ &$(1, 2)$, $(2, 1)$ &$(1, 2, 3)$ & $\PSU(2)_6$\cite{B1}\\ \hline
$1$, $2$ &$(1, 2)$, $(1, 3)$, $(2, 1)$, $(3, 2)$ &$(1, 2, 3)$ & $\PSU(2)_5$\cite{B1} \\ \hline
$1$, $2$, $3$ &$(1, 2)$, $(1, 3)$ &$(1, 2, 3)$ & $\Fib^2$ \cite{RSW}\\ \hline
\end{longtable}

\begin{longtable}{|c| M | M | c|}
\caption{Rank 5 Fusion Rings\label{T:Rank5}}\\
\hline
Loops & Arcs & Hyperedges & Grothendick Class \\ \hline
 \endfirsthead
\caption{Rank 5 Fusion Rings (continued)}\\ \hline
Loops & Arcs & Hyperedges & Grothendick Class \\ \hline
\endhead
&$(1, 2)$, $(1, 3)$, $(1, 4)$ &$(2, 3, 4)$ &  $\mathrm{TY}(\nicefrac{\Z}{2\Z} \times \nicefrac{\Z}{2\Z},\chi,\nu)$ \\ \hline
&$(1, 3)$, $(1, 4)$, $(2, 1)$, $(2, 4)$, $(3, 2)$, $(3, 4)$ &$(1, 2, 3)$ &  $\mathcal{C}(\so_{7},e^{\nicefrac{\pi i}{14}},14)_{ad}$\\ \hline
$1$ &$(1, 2)$, $(1, 3)$, $(1, 4)$ &$(2, 3, 4)$ &  not categorifiable\cite{ENO1}\\ \hline
$1$ &$(1, 4)$, $(2, 1)$, $(3, 1)$ &$(1, 2, 3)$, $(2, 3, 4)$ &  $\mathcal{C}(\so_{3},e^{\nicefrac{\pi i}{6}},6)$\\ \hline
$1$, $3$ &$(1, 4)$, $(2, 1)$, $(2, 3)$, $(3, 1)$ &$(1, 2, 3)$, $(2, 3, 4)$ & \\ \hline
$1$, $3$ &$(1, 2)$, $(1, 3)$, $(1, 4)$, $(2, 1)$, $(2, 3)$, $(3, 1)$ &$(1, 2, 3)$, $(2, 3, 4)$ & \\ \hline
$2$, $3$ &$(1, 2)$, $(1, 3)$, $(1, 4)$, $(2, 1)$, $(2, 3)$, $(3, 1)$, $(3, 2)$ &$(1, 2, 3)$, $(2, 3, 4)$ & \\ \hline
$1$, $2$, $3$ &$(1, 2)$, $(1, 3)$, $(1, 4)$, $(2, 1)$, $(2, 3)$, $(3, 1)$, $(4, 3)$ &$(1, 2, 3)$, $(1, 2, 4)$, $(2, 3, 4)$ & \\ \hline
$1$, $2$, $3$ &$(1, 4)$, $(2, 1)$, $(2, 3)$, $(3, 1)$, $(3, 2)$ &$(1, 2, 3)$, $(2, 3, 4)$ & \\ \hline
$1$, $2$, $3$, $4$ &$(1, 2)$, $(1, 3)$, $(1, 4)$, $(2, 4)$, $(3, 2)$, $(4, 3)$ &$(1, 2, 3)$, $(1, 2, 4)$, $(1, 3, 4)$ & \\ \hline
\end{longtable}

\begin{longtable}{|c| M | M | c|}
\caption{Rank 6 Fusion Rings\label{T:Rank6}}\\
\hline
Loops & Arcs & Hyperedges & Grothendick Class \\ \hline
 \endfirsthead
\caption{Rank 6 Fusion Rings (continued)}\\ \hline
Loops & Arcs & Hyperedges & Grothendick Class \\ \hline
\endhead
&$(1, 5)$, $(2, 5)$ &$(1, 2, 3)$, $(1, 2, 4)$, $(3, 4, 5)$ & 
$\Sem\boxtimes\Ising$ \cite{RSW}\\\hline
&$(1, 2)$, $(1, 5)$, $(2, 1)$, $(2, 5)$, $(3, 1)$, $(3, 2)$, $(4, 1)$, $(4, 2)$ &$(1, 3, 4)$, $(2, 3, 4)$, $(3, 4, 5)$ & \\ \hline
$2$ &$(1, 2)$, $(1, 5)$, $(2, 5)$ &$(1, 2, 3)$, $(1, 2, 4)$, $(3, 4, 5)$ & 
$\Sem\boxtimes \Rep\paren{\mfS_{3}}$\cite{RSW,O3}\\\hline
$4$ &$(1, 3)$, $(1, 4)$, $(1, 5)$, $(2, 5)$, $(3, 4)$ &$(1, 2, 3)$, $(1, 2, 4)$,
$(3, 4, 5)$ & $\Fib\boxtimes\Ising$\cite{RSW}\\\hline
$2$ &$(1, 2)$, $(1, 4)$, $(2, 4)$, $(3, 2)$, $(4, 2)$ &$(1, 2, 3)$, $(1, 2, 5)$, $(1, 3, 4)$, $(3, 4, 5)$ & $\Sem\boxtimes\PSU\paren{3}_{2}$\cite{RSW}\\\hline
$1$ &$(1, 5)$, $(2, 4)$, $(2, 5)$, $(3, 2)$, $(3, 5)$, $(4, 3)$, $(4, 5)$ &$(1, 2, 3)$, $(1, 2, 4)$, $(1, 3, 4)$ & \\ \hline 
$2$ &$(1, 2)$, $(1, 3)$, $(1, 4)$, $(2, 3)$, $(2, 4)$, $(3, 4)$, $(3, 5)$, $(4, 3)$, $(4, 5)$ &$(1, 2, 3)$, $(1, 2, 4)$, $(1, 2, 5)$ & \\ \hline
$1$ &$(1, 5)$, $(2, 3)$, $(2, 4)$, $(2, 5)$, $(3, 2)$, $(3, 4)$, $(3, 5)$, $(4, 2)$, $(4, 3)$, $(4, 5)$ &$(1, 2, 3)$, $(1, 2, 4)$, $(1, 3, 4)$ & \\ \hline
$1$ &$(1, 2)$, $(1, 3)$, $(1, 4)$, $(1, 5)$, $(2, 3)$, $(2, 4)$, $(2, 5)$, $(3, 2)$, $(3, 4)$, $(3, 5)$, $(4, 2)$, $(4, 3)$, $(4, 5)$, $(5, 2)$, $(5, 3)$, $(5, 4)$ &$(1, 2, 3)$, $(1, 2, 4)$, $(1, 2, 5)$, $(1, 3, 4)$, $(1, 3, 5)$, $(1, 4, 5)$ & \\ \hline
$1$, $2$ &$(1, 2)$, $(1, 5)$, $(2, 1)$, $(2, 5)$ &$(1, 2, 3)$, $(1, 2, 4)$, $(3, 4, 5)$ & \\ \hline
$1$, $2$ &$(1, 2)$, $(1, 4)$, $(2, 1)$, $(2, 4)$, $(3, 1)$, $(3, 2)$, $(4, 1)$, $(4, 2)$ &$(1, 2, 3)$, $(1, 2, 5)$, $(1, 3, 4)$, $(2, 3, 4)$, $(3, 4, 5)$ & \\ \hline
$1$, $2$ &$(1, 2)$, $(1, 3)$, $(1, 4)$, $(2, 1)$, $(2, 3)$, $(2, 4)$, $(3, 4)$, $(3, 5)$, $(4, 3)$, $(4, 5)$ &$(1, 2, 3)$, $(1, 2, 4)$, $(1, 2, 5)$ & \\ \hline 
$1$, $2$, $4$ &$(1, 2)$, $(1, 3)$, $(1, 4)$, $(1, 5)$, $(2, 5)$, $(3, 4)$ &$(1, 2, 3)$, $(1, 2, 4)$, $(3, 4, 5)$ & $\Fib \boxtimes \Rep\paren{\mfS_{3}}$\cite{O3}\\ \hline
$1$, $3$, $4$ &$(1, 2)$, $(1, 3)$, $(1, 4)$, $(1, 5)$, $(2, 1)$, $(2, 3)$, $(2, 4)$, $(3, 5)$, $(5, 3)$ &$(1, 2, 3)$, $(1, 2, 5)$, $(1, 3, 4)$, $(2, 4, 5)$ & $\Fib \boxtimes \PSU(3)_2$ \cite{O3}\\ \hline 
$1$, $2$, $4$ &$(1, 2)$, $(1, 3)$, $(1, 4)$, $(2, 1)$, $(2, 3)$, $(2, 4)$, $(3, 2)$, $(3, 4)$, $(4, 2)$ &$(1, 2, 3)$, $(1, 2, 4)$, $(1, 2, 5)$, $(1, 3, 4)$, $(3, 4, 5)$ & \\ \hline
$1$, $2$, $4$ &$(1, 2)$, $(1, 3)$, $(1, 4)$, $(1, 5)$, $(2, 1)$, $(2, 3)$, $(2, 4)$, $(3, 1)$, $(3, 2)$, $(3, 4)$, $(4, 2)$, $(5, 4)$ &$(1, 2, 3)$, $(1, 2, 4)$, $(1, 2, 5)$, $(1, 3, 4)$, $(2, 3, 5)$, $(3, 4, 5)$ & \\ \hline
$1$, $2$, $3$, $4$ &$(1, 5)$, $(2, 5)$, $(3, 5)$, $(4, 5)$ &$(1, 2, 3)$, $(1, 2, 4)$, $(1, 3, 4)$, $(2, 3, 4)$ & \\ \hline
$1$, $2$, $3$, $4$ &$(1, 5)$, $(2, 4)$, $(2, 5)$, $(3, 2)$, $(3, 5)$, $(4, 3)$, $(4, 5)$ &$(1, 2, 3)$, $(1, 2, 4)$, $(1, 3, 4)$, $(2, 3, 4)$ & \\ \hline
$1$, $2$, $3$, $4$ &$(1, 3)$, $(1, 4)$, $(1, 5)$, $(2, 1)$, $(2, 5)$, $(3, 1)$, $(3, 2)$, $(3, 4)$, $(4, 1)$, $(4, 2)$, $(4, 3)$ &$(1, 2, 3)$, $(1, 2, 4)$, $(1, 3, 4)$, $(2, 3, 4)$, $(3, 4, 5)$ & \\ \hline
$1$, $2$, $4$, $5$ &$(1, 2)$, $(1, 3)$, $(1, 4)$, $(1, 5)$, $(2, 3)$, $(2, 5)$, $(3, 2)$, $(3, 4)$, $(3, 5)$, $(4, 2)$, $(4, 3)$, $(5, 3)$, $(5, 4)$ &$(1, 2, 3)$, $(1, 2, 4)$, $(1, 2, 5)$, $(1, 3, 4)$, $(1, 3, 5)$, $(1, 4, 5)$, $(2, 4, 5)$ & \\ \hline
\end{longtable}

\begin{longtable}{|c| M | M | c|}
\caption{Rank 7 Fusion Rings\label{T:Rank7}}\\
\hline
Loops & Arcs & Hyperedges & Grothendick Class \\ \hline
 \endfirsthead
\caption{Rank 7 Fusion Rings (continued)}\\ \hline
Loops & Arcs & Hyperedges & Grothendick Class \\ \hline
\endhead
&$(1, 4)$, $(1, 5)$, $(1, 6)$, $(2, 1)$, $(2, 6)$, $(3, 1)$, $(3, 6)$ &$(1, 2, 3)$, $(2, 3, 4)$, $(2, 3, 5)$, $(4, 5, 6)$ & \\ \hline
&$(1, 5)$, $(1, 6)$, $(2, 1)$, $(2, 6)$, $(3, 2)$, $(3, 6)$, $(4, 3)$, $(4, 6)$, $(5, 4)$, $(5, 6)$ &$(1, 2, 4)$, $(1, 3, 4)$, $(1, 3, 5)$, $(2, 3, 5)$, $(2, 4, 5)$ & \\ \hline
&$(1, 3)$, $(1, 4)$, $(1, 5)$, $(2, 3)$, $(2, 4)$, $(2, 5)$, $(3, 5)$, $(3, 6)$, $(4, 3)$, $(4, 6)$, $(5, 4)$, $(5, 6)$ &$(1, 2, 3)$, $(1, 2, 4)$, $(1, 2, 5)$, $(1, 2, 6)$, $(3, 4, 5)$ & \\ \hline
$3$ &$(1, 4)$, $(1, 5)$, $(1, 6)$, $(2, 1)$, $(2, 3)$, $(2, 6)$, $(3, 1)$, $(3, 6)$ &$(1, 2, 3)$, $(2, 3, 4)$, $(2, 3, 5)$, $(4, 5, 6)$ & \\ \hline
$1$ &$(1, 2)$, $(1, 3)$, $(1, 4)$, $(1, 5)$, $(1, 6)$, $(2, 1)$, $(2, 6)$, $(3, 1)$, $(3, 6)$ &$(1, 2, 3)$, $(2, 3, 4)$, $(2, 3, 5)$, $(4, 5, 6)$ & \\ \hline 
$3$ &$(1, 2)$, $(1, 3)$, $(1, 4)$, $(1, 5)$, $(1, 6)$, $(2, 1)$, $(2, 3)$, $(2, 6)$, $(3, 1)$, $(3, 6)$ &$(1, 2, 3)$, $(2, 3, 4)$, $(2, 3, 5)$, $(4, 5, 6)$ & \\ \hline
$2$ &$(1, 2)$, $(1, 3)$, $(1, 4)$, $(1, 5)$, $(2, 3)$, $(2, 4)$, $(2, 5)$, $(3, 5)$, $(3, 6)$, $(4, 3)$, $(4, 6)$, $(5, 4)$, $(5, 6)$ &$(1, 2, 3)$, $(1, 2, 4)$, $(1, 2, 5)$, $(1, 2, 6)$, $(3, 4, 5)$ & \\ \hline
$2$, $3$ &$(1, 2)$, $(1, 3)$, $(1, 6)$, $(2, 3)$, $(3, 2)$, $(4, 3)$, $(5, 3)$ &$(1, 2, 4)$, $(1, 2, 5)$, $(1, 3, 4)$, $(1, 3, 5)$, $(2, 3, 6)$, $(2, 4, 5)$, $(4, 5, 6)$ & \\ \hline
$2$, $3$ &$(1, 4)$, $(1, 5)$, $(1, 6)$, $(2, 1)$, $(2, 3)$, $(2, 6)$, $(3, 1)$, $(3, 2)$, $(3, 6)$ &$(1, 2, 3)$, $(2, 3, 4)$, $(2, 3, 5)$, $(4, 5, 6)$ & \\ \hline
$1$, $5$ &$(1, 6)$, $(2, 1)$, $(2, 4)$, $(2, 5)$, $(3, 1)$, $(3, 4)$, $(3, 5)$, $(4, 1)$, $(4, 3)$, $(4, 5)$, $(5, 1)$, $(5, 3)$ &$(1, 2, 3)$, $(1, 4, 5)$, $(2, 3, 4)$, $(2, 3, 5)$, $(2, 3, 6)$, $(2, 4, 5)$, $(4, 5, 6)$ & \\ \hline
$1$, $2$ &$(1, 2)$, $(1, 3)$, $(1, 4)$, $(1, 5)$, $(2, 1)$, $(2, 3)$, $(2, 4)$, $(2, 5)$, $(3, 5)$, $(3, 6)$, $(4, 3)$, $(4, 6)$, $(5, 4)$, $(5, 6)$ &$(1, 2, 3)$, $(1, 2, 4)$, $(1, 2, 5)$, $(1, 2, 6)$, $(3, 4, 5)$ & \\ \hline
$3$, $4$, $5$ &$(1, 2)$, $(1, 3)$, $(1, 4)$, $(1, 5)$, $(1, 6)$, $(2, 3)$, $(2, 5)$, $(3, 5)$, $(4, 5)$, $(5, 4)$ &$(1, 2, 4)$, $(1, 2, 5)$, $(1, 3, 4)$, $(1, 3, 5)$, $(2, 3, 4)$, $(2, 3, 6)$, $(4, 5, 6)$ & \\ \hline
$1$, $2$, $3$ &$(1, 6)$, $(2, 1)$, $(2, 3)$, $(2, 4)$, $(2, 5)$, $(3, 1)$, $(3, 2)$, $(3, 4)$, $(3, 5)$, $(4, 1)$, $(4, 3)$, $(5, 1)$, $(5, 3)$ &$(1, 2, 3)$, $(1, 4, 5)$, $(2, 3, 4)$, $(2, 3, 5)$, $(2, 3, 6)$, $(2, 4, 5)$, $(4, 5, 6)$ & \\ \hline
$1$, $2$, $3$, $5$ &$(1, 2)$, $(1, 3)$, $(1, 4)$, $(1, 5)$, $(1, 6)$, $(2, 1)$, $(2, 3)$, $(2, 5)$, $(3, 1)$, $(3, 2)$, $(3, 5)$, $(4, 3)$, $(4, 5)$, $(5, 3)$ &$(1, 2, 3)$, $(1, 2, 4)$, $(1, 2, 5)$, $(1, 3, 4)$, $(1, 3, 5)$, $(2, 3, 4)$, $(2, 3, 6)$, $(2, 4, 5)$, $(4, 5, 6)$ & \\ \hline
$1$, $2$, $3$, $5$ &$(1, 2)$, $(1, 3)$, $(1, 4)$, $(1, 5)$, $(1, 6)$, $(2, 1)$, $(2, 3)$, $(2, 4)$, $(2, 5)$, $(3, 1)$, $(3, 2)$, $(3, 5)$, $(4, 1)$, $(4, 3)$, $(4, 5)$, $(5, 3)$, $(6, 5)$ &$(1, 2, 3)$, $(1, 2, 4)$, $(1, 2, 5)$, $(1, 2, 6)$, $(1, 3, 4)$, $(1, 3, 5)$, $(2, 3, 4)$, $(2, 3, 6)$, $(2, 4, 5)$, $(3, 4, 6)$, $(4, 5, 6)$ & \\ \hline
$1$, $2$, $3$, $5$ &$(1, 2)$, $(1, 3)$, $(1, 4)$, $(1, 6)$, $(2, 1)$, $(2, 3)$, $(2, 4)$, $(2, 5)$, $(3, 1)$, $(3, 2)$, $(3, 4)$, $(3, 5)$, $(4, 2)$, $(4, 3)$, $(4, 5)$, $(4, 6)$, $(5, 4)$, $(5, 6)$ &$(1, 2, 3)$, $(1, 2, 4)$, $(1, 2, 5)$, $(1, 3, 4)$, $(1, 3, 5)$, $(1, 4, 5)$, $(2, 3, 4)$, $(2, 3, 5)$, $(2, 3, 6)$ & \\ \hline
$1$, $2$, $3$, $4$, $5$ &$(1, 4)$, $(1, 5)$, $(1, 6)$, $(2, 1)$, $(2, 3)$, $(2, 4)$, $(2, 5)$, $(3, 1)$, $(3, 2)$, $(3, 4)$, $(3, 5)$, $(4, 3)$, $(4, 5)$, $(5, 3)$, $(5, 4)$ &$(1, 2, 3)$, $(1, 2, 4)$, $(1, 2, 5)$, $(1, 3, 4)$, $(1, 3, 5)$, $(2, 3, 4)$, $(2, 3, 5)$, $(2, 3, 6)$, $(2, 4, 5)$, $(4, 5, 6)$ & \\ \hline
$1$, $2$, $3$, $4$, $5$, $6$ &$(1, 2)$, $(1, 3)$, $(1, 4)$, $(1, 5)$, $(1, 6)$, $(2, 3)$, $(2, 4)$, $(2, 6)$, $(3, 2)$, $(3, 5)$, $(3, 6)$, $(4, 2)$, $(4, 3)$, $(4, 5)$, $(5, 2)$, $(5, 4)$, $(5, 6)$, $(6, 3)$, $(6, 4)$, $(6, 5)$ &$(1, 2, 3)$, $(1, 2, 4)$, $(1, 2, 5)$, $(1, 2, 6)$, $(1, 3, 4)$, $(1, 3, 5)$, $(1, 3, 6)$, $(1, 4, 5)$, $(1, 4, 6)$, $(1, 5, 6)$, $(2, 3, 5)$, $(2, 4, 6)$, $(2, 5, 6)$, $(3, 4, 5)$, $(3, 4, 6)$ & \\ \hline
\end{longtable}

\begin{longtable}{|c| M | M | c|}
\caption{Rank 8 Fusion Rings\label{T:Rank8}}\\ 
\hline
Loops & Arcs & Hyperedges & Grothendick Class \\ \hline 
\endfirsthead 
\caption{Rank 8 Fusion Rings (continued)}\\ \hline 
Loops & Arcs & Hyperedges & Grothendick Class \\ \hline 
\endhead 
 &  & $(1, 2, 3)$, $(1, 4, 5)$, $(1, 6, 7)$, $(2, 4, 6)$, $(2, 5, 7)$, $(3, 4, 7)$, $(3, 5, 6)$ & $\Sem^{3}$\cite{RSW}\\ \hline 

 & $(1, 6)$, $(1, 7)$, $(2, 6)$, $(2, 7)$, $(5, 1)$, $(5, 7)$, $(6, 1)$, $(6, 7)$ & $(1, 2, 3)$, $(1, 2, 4)$, $(1, 2, 5)$, $(2, 5, 6)$, $(3, 4, 7)$, $(3, 5, 6)$, $(4, 5, 6)$ & $\Sem\boxtimes \PSO(5)_6$ \cite{RSW,B1} \\ \hline 

 & $(1, 4)$, $(1, 7)$, $(2, 6)$, $(2, 7)$, $(3, 5)$, $(3, 7)$, $(4, 2)$, $(4, 7)$, $(5, 1)$, $(5, 7)$, $(6, 3)$, $(6, 7)$ & $(1, 2, 3)$, $(1, 2, 6)$, $(1, 3, 4)$, $(1, 5, 6)$, $(2, 3, 5)$, $(2, 4, 5)$, $(3, 4, 6)$, $(4, 5, 6)$ & \\ \hline 

 & $(1, 4)$, $(1, 5)$, $(1, 6)$, $(1, 7)$, $(2, 4)$, $(2, 5)$, $(2, 6)$, $(2, 7)$, $(4, 1)$, $(4, 6)$, $(4, 7)$, $(5, 1)$, $(5, 6)$, $(5, 7)$, $(6, 3)$, $(6, 7)$, $(7, 3)$, $(7, 6)$ & $(1, 2, 3)$, $(1, 2, 4)$, $(1, 2, 5)$, $(1, 2, 6)$, $(1, 2, 7)$, $(2, 4, 5)$, $(3, 4, 5)$, $(4, 5, 6)$, $(4, 5, 7)$ & \\ \hline 

$7$ & $(1, 7)$, $(2, 7)$, $(3, 7)$ & $(1, 2, 3)$, $(1, 2, 4)$, $(1, 3, 5)$, $(1, 6, 7)$, $(2, 3, 6)$, $(2, 5, 7)$, $(3, 4, 7)$, $(4, 5, 6)$ & $\Fib\boxtimes\Sem^{2}$\cite{RSW} \\ \hline 

$7$ & $(1, 5)$, $(1, 6)$, $(1, 7)$, $(2, 5)$, $(2, 6)$, $(2, 7)$, $(6, 3)$, $(6, 7)$, $(7, 3)$ & $(1, 2, 3)$, $(1, 2, 4)$, $(1, 2, 6)$, $(1, 2, 7)$, $(3, 4, 5)$, $(4, 6, 7)$, $(5, 6, 7)$ & \\ \hline 

$4$ & $(2, 4)$, $(3, 1)$, $(3, 5)$, $(5, 1)$, $(5, 3)$, $(6, 1)$, $(6, 2)$, $(6, 3)$, $(6, 4)$, $(6, 7)$, $(7, 1)$, $(7, 2)$, $(7, 4)$, $(7, 5)$, $(7, 6)$ & $(1, 2, 4)$, $(2, 3, 7)$, $(2, 5, 6)$, $(3, 4, 7)$, $(3, 6, 7)$, $(4, 5, 6)$, $(5, 6, 7)$ & $\Fib \boxtimes \PSO(5)_6$\cite{RSW,B1} \\ \hline 

$5$ & $(1, 4)$, $(1, 5)$, $(1, 6)$, $(1, 7)$, $(2, 4)$, $(2, 5)$, $(2, 6)$, $(2, 7)$, $(4, 3)$, $(4, 6)$, $(5, 3)$, $(6, 3)$, $(6, 7)$, $(7, 3)$, $(7, 4)$ & $(1, 2, 3)$, $(1, 2, 4)$, $(1, 2, 5)$, $(1, 2, 6)$, $(1, 2, 7)$, $(4, 5, 6)$, $(4, 5, 7)$, $(5, 6, 7)$ & \\ \hline 

$6$, $7$ & $(1, 6)$, $(1, 7)$, $(2, 6)$, $(2, 7)$, $(6, 7)$, $(7, 6)$ & $(1, 2, 3)$, $(1, 2, 6)$, $(1, 2, 7)$, $(1, 4, 6)$, $(1, 5, 7)$, $(2, 4, 7)$, $(2, 5, 6)$, $(3, 4, 5)$, $(3, 6, 7)$ & $\Sem\boxtimes\PSU(2)_{6}$\cite{RSW,B1}\\ \hline 

$6$, $7$ & $(1, 4)$, $(1, 5)$, $(1, 6)$, $(1, 7)$, $(5, 2)$, $(5, 6)$, $(6, 2)$, $(7, 4)$, $(7, 5)$, $(7, 6)$ & $(1, 2, 7)$, $(1, 3, 7)$, $(1, 5, 7)$, $(1, 6, 7)$, $(2, 3, 4)$, $(3, 5, 6)$, $(4, 5, 6)$ & \\ \hline 

$6$, $7$ & $(1, 5)$, $(1, 6)$, $(1, 7)$, $(2, 6)$, $(2, 7)$, $(3, 7)$, $(5, 7)$, $(6, 5)$, $(6, 7)$, $(7, 6)$ & $(1, 2, 5)$, $(1, 2, 6)$, $(1, 2, 7)$, $(1, 3, 6)$, $(1, 3, 7)$, $(1, 4, 6)$, $(2, 3, 5)$, $(2, 3, 6)$, $(2, 4, 7)$, $(3, 4, 5)$, $(5, 6, 7)$ & $\Sem\boxtimes\PSU(2)_{5}$\cite{RSW,B1} \\ \hline 

$5$, $7$ & $(1, 4)$, $(1, 5)$, $(1, 6)$, $(1, 7)$, $(4, 1)$, $(4, 5)$, $(4, 6)$, $(5, 1)$, $(5, 6)$, $(7, 4)$, $(7, 5)$, $(7, 6)$ & $(1, 2, 7)$, $(1, 3, 7)$, $(1, 4, 7)$, $(1, 5, 7)$, $(2, 3, 6)$, $(2, 4, 5)$, $(3, 4, 5)$, $(4, 5, 7)$ & \\ \hline 

$4$, $7$ & $(1, 3)$, $(1, 4)$, $(1, 5)$, $(1, 6)$, $(1, 7)$, $(3, 2)$, $(3, 5)$, $(4, 2)$, $(5, 2)$, $(5, 6)$, $(6, 2)$, $(6, 3)$, $(7, 3)$, $(7, 4)$, $(7, 5)$, $(7, 6)$ & $(1, 2, 7)$, $(1, 3, 7)$, $(1, 4, 7)$, $(1, 5, 7)$, $(1, 6, 7)$, $(3, 4, 5)$, $(3, 4, 6)$, $(4, 5, 6)$ & \\ \hline 

$5$, $6$, $7$ & $(1, 5)$, $(1, 6)$, $(1, 7)$, $(2, 5)$, $(3, 6)$, $(7, 5)$, $(7, 6)$ & $(1, 2, 6)$, $(1, 2, 7)$, $(1, 3, 5)$, $(1, 3, 7)$, $(1, 4, 7)$, $(2, 3, 7)$, $(2, 4, 5)$, $(3, 4, 6)$, $(5, 6, 7)$ & $\Fib^2\boxtimes\Sem$\cite{RSW} \\ \hline 

$1$, $5$, $6$ & $(1, 5)$, $(1, 6)$, $(5, 1)$, $(5, 6)$, $(6, 1)$, $(6, 5)$, $(7, 1)$, $(7, 5)$, $(7, 6)$ & $(1, 2, 5)$, $(1, 3, 6)$, $(1, 4, 7)$, $(1, 5, 7)$, $(1, 6, 7)$, $(2, 3, 4)$, $(2, 6, 7)$, $(3, 5, 7)$, $(4, 5, 6)$, $(5, 6, 7)$ & \\ \hline 

$1$, $6$, $7$ & $(1, 4)$, $(1, 5)$, $(1, 6)$, $(1, 7)$, $(5, 2)$, $(5, 6)$, $(6, 2)$, $(7, 1)$, $(7, 4)$, $(7, 5)$, $(7, 6)$ & $(1, 2, 7)$, $(1, 3, 7)$, $(1, 5, 7)$, $(1, 6, 7)$, $(2, 3, 4)$, $(3, 5, 6)$, $(4, 5, 6)$ & \\ \hline 

$1$, $6$, $7$ & $(1, 4)$, $(1, 5)$, $(1, 6)$, $(1, 7)$, $(5, 1)$, $(5, 2)$, $(5, 6)$, $(5, 7)$, $(6, 1)$, $(6, 2)$, $(6, 7)$, $(7, 1)$, $(7, 4)$, $(7, 5)$, $(7, 6)$ & $(1, 2, 7)$, $(1, 3, 7)$, $(1, 5, 6)$, $(1, 5, 7)$, $(1, 6, 7)$, $(2, 3, 4)$, $(3, 5, 6)$, $(4, 5, 6)$, $(5, 6, 7)$ & \\ \hline 

$1$, $4$, $7$ & $(1, 3)$, $(1, 4)$, $(1, 5)$, $(1, 6)$, $(1, 7)$, $(3, 2)$, $(3, 5)$, $(4, 2)$, $(5, 2)$, $(5, 6)$, $(6, 2)$, $(6, 3)$, $(7, 1)$, $(7, 3)$, $(7, 4)$, $(7, 5)$, $(7, 6)$ & $(1, 2, 7)$, $(1, 3, 7)$, $(1, 4, 7)$, $(1, 5, 7)$, $(1, 6, 7)$, $(3, 4, 5)$, $(3, 4, 6)$, $(4, 5, 6)$ & \\ \hline 

$1$, $4$, $7$ & $(1, 4)$, $(1, 7)$, $(3, 1)$, $(3, 2)$, $(3, 5)$, $(3, 7)$, $(4, 2)$, $(5, 1)$, $(5, 2)$, $(5, 6)$, $(5, 7)$, $(6, 1)$, $(6, 2)$, $(6, 3)$, $(6, 7)$, $(7, 1)$, $(7, 4)$ & $(1, 2, 7)$, $(1, 3, 5)$, $(1, 3, 6)$, $(1, 4, 7)$, $(1, 5, 6)$, $(3, 4, 5)$, $(3, 4, 6)$, $(3, 5, 7)$, $(3, 6, 7)$, $(4, 5, 6)$, $(5, 6, 7)$ & \\ \hline 

$4$, $5$, $6$ & $(1, 4)$, $(1, 5)$, $(2, 4)$, $(2, 5)$, $(4, 5)$, $(4, 6)$, $(5, 4)$, $(5, 6)$, $(6, 1)$, $(6, 2)$, $(6, 4)$, $(6, 5)$, $(7, 1)$, $(7, 2)$, $(7, 4)$, $(7, 5)$, $(7, 6)$ & $(1, 2, 3)$, $(1, 2, 4)$, $(1, 2, 5)$, $(1, 4, 6)$, $(1, 5, 7)$, $(1, 6, 7)$, $(2, 4, 7)$, $(2, 5, 6)$, $(2, 6, 7)$, $(3, 4, 5)$, $(3, 6, 7)$, $(4, 5, 7)$, $(4, 6, 7)$, $(5, 6, 7)$ & \\ \hline 

$1$, $4$, $7$ & $(1, 4)$, $(1, 7)$, $(3, 1)$, $(3, 2)$, $(3, 5)$, $(3, 6)$, $(3, 7)$, $(4, 2)$, $(5, 1)$, $(5, 2)$, $(5, 3)$, $(5, 6)$, $(5, 7)$, $(6, 1)$, $(6, 2)$, $(6, 3)$, $(6, 5)$, $(6, 7)$, $(7, 1)$, $(7, 4)$ & $(1, 2, 7)$, $(1, 3, 5)$, $(1, 3, 6)$, $(1, 4, 7)$, $(1, 5, 6)$, $(3, 4, 5)$, $(3, 4, 6)$, $(3, 5, 7)$, $(3, 6, 7)$, $(4, 5, 6)$, $(5, 6, 7)$ & \\ \hline 

$1$, $3$, $5$ & $(1, 3)$, $(1, 4)$, $(1, 5)$, $(1, 6)$, $(3, 1)$, $(3, 4)$, $(3, 5)$, $(3, 7)$, $(4, 1)$, $(4, 3)$, $(4, 5)$, $(5, 1)$, $(5, 3)$, $(6, 1)$, $(6, 3)$, $(6, 4)$, $(6, 5)$, $(6, 7)$, $(7, 1)$, $(7, 3)$, $(7, 4)$, $(7, 5)$, $(7, 6)$ & $(1, 2, 7)$, $(1, 3, 5)$, $(1, 3, 6)$, $(1, 3, 7)$, $(1, 4, 6)$, $(1, 4, 7)$, $(1, 5, 7)$, $(1, 6, 7)$, $(2, 3, 6)$, $(2, 4, 5)$, $(3, 4, 6)$, $(3, 4, 7)$, $(3, 5, 6)$, $(3, 6, 7)$, $(4, 5, 6)$, $(4, 5, 7)$, $(5, 6, 7)$ & \\ \hline 

$1$, $4$, $7$ & $(1, 3)$, $(1, 4)$, $(1, 5)$, $(1, 6)$, $(1, 7)$, $(3, 1)$, $(3, 2)$, $(3, 5)$, $(3, 6)$, $(3, 7)$, $(4, 2)$, $(5, 1)$, $(5, 2)$, $(5, 3)$, $(5, 6)$, $(5, 7)$, $(6, 1)$, $(6, 2)$, $(6, 3)$, $(6, 5)$, $(6, 7)$, $(7, 1)$, $(7, 3)$, $(7, 4)$, $(7, 5)$, $(7, 6)$ & $(1, 2, 7)$, $(1, 3, 5)$, $(1, 3, 6)$, $(1, 3, 7)$, $(1, 4, 7)$, $(1, 5, 6)$, $(1, 5, 7)$, $(1, 6, 7)$, $(3, 4, 5)$, $(3, 4, 6)$, $(3, 5, 7)$, $(3, 6, 7)$, $(4, 5, 6)$, $(5, 6, 7)$ & \\ \hline 

$1$, $2$, $3$, $4$ & $(1, 5)$, $(2, 5)$, $(3, 5)$, $(4, 5)$, $(6, 1)$, $(6, 2)$, $(6, 3)$, $(6, 4)$, $(7, 1)$, $(7, 2)$, $(7, 3)$, $(7, 4)$ & $(1, 2, 3)$, $(1, 2, 4)$, $(1, 3, 4)$, $(1, 6, 7)$, $(2, 3, 4)$, $(2, 6, 7)$, $(3, 6, 7)$, $(4, 6, 7)$, $(5, 6, 7)$ & \\ \hline 

$1$, $5$, $6$, $7$ & $(1, 5)$, $(1, 6)$, $(1, 7)$, $(5, 1)$, $(5, 6)$, $(5, 7)$, $(6, 1)$, $(6, 5)$, $(6, 7)$, $(7, 1)$, $(7, 5)$, $(7, 6)$ & $(1, 2, 5)$, $(1, 3, 6)$, $(1, 4, 7)$, $(1, 5, 6)$, $(1, 5, 7)$, $(1, 6, 7)$, $(2, 3, 4)$, $(2, 6, 7)$, $(3, 5, 7)$, $(4, 5, 6)$, $(5, 6, 7)$ & \\ \hline 

$4$, $5$, $6$, $7$ & $(1, 4)$, $(1, 5)$, $(1, 6)$, $(1, 7)$, $(2, 6)$, $(2, 7)$, $(4, 1)$, $(4, 2)$, $(4, 5)$, $(4, 6)$, $(4, 7)$, $(5, 1)$, $(5, 2)$, $(5, 4)$, $(5, 6)$, $(5, 7)$, $(6, 4)$, $(6, 5)$, $(6, 7)$, $(7, 6)$ & $(1, 2, 4)$, $(1, 2, 6)$, $(1, 2, 7)$, $(1, 3, 6)$, $(1, 4, 5)$, $(1, 4, 6)$, $(1, 5, 6)$, $(1, 5, 7)$, $(2, 3, 7)$, $(2, 4, 5)$, $(2, 5, 6)$, $(3, 4, 5)$, $(4, 5, 6)$, $(4, 5, 7)$, $(4, 6, 7)$ & \\ \hline 

$1$, $2$, $3$, $5$ & $(1, 3)$, $(1, 4)$, $(1, 5)$, $(1, 6)$, $(1, 7)$, $(2, 4)$, $(2, 6)$, $(3, 1)$, $(3, 2)$, $(3, 5)$, $(3, 6)$, $(3, 7)$, $(5, 1)$, $(5, 2)$, $(5, 3)$, $(5, 6)$, $(5, 7)$, $(6, 1)$, $(6, 2)$, $(6, 4)$, $(6, 7)$, $(7, 1)$, $(7, 2)$, $(7, 3)$, $(7, 4)$, $(7, 5)$ & $(1, 2, 3)$, $(1, 2, 5)$, $(1, 2, 7)$, $(1, 3, 5)$, $(1, 3, 6)$, $(1, 3, 7)$, $(1, 5, 6)$, $(1, 5, 7)$, $(2, 3, 5)$, $(2, 6, 7)$, $(3, 4, 5)$, $(3, 5, 6)$, $(3, 5, 7)$, $(3, 6, 7)$, $(5, 6, 7)$ & \\ \hline 

$1$, $2$, $4$, $6$, $7$ & $(1, 4)$, $(4, 1)$, $(5, 2)$, $(6, 1)$, $(6, 2)$, $(6, 4)$, $(6, 7)$, $(7, 1)$, $(7, 2)$, $(7, 4)$, $(7, 6)$ & $(1, 2, 6)$, $(1, 3, 4)$, $(1, 5, 7)$, $(1, 6, 7)$, $(2, 3, 5)$, $(2, 4, 7)$, $(3, 6, 7)$, $(4, 5, 6)$, $(4, 6, 7)$, $(5, 6, 7)$ & $\Fib \boxtimes \PSU(2)_6$ \cite{RSW,B1} \\ \hline 

$1$, $2$, $3$, $4$, $7$ & $(1, 6)$, $(2, 6)$, $(3, 6)$, $(4, 6)$, $(5, 1)$, $(5, 2)$, $(5, 3)$, $(5, 4)$, $(5, 7)$, $(7, 1)$, $(7, 2)$, $(7, 3)$, $(7, 4)$ & $(1, 2, 3)$, $(1, 2, 4)$, $(1, 3, 4)$, $(1, 5, 7)$, $(2, 3, 4)$, $(2, 5, 7)$, $(3, 5, 7)$, $(4, 5, 7)$, $(5, 6, 7)$ & \\ \hline 

$1$, $4$, $5$, $6$, $7$ & $(1, 3)$, $(1, 5)$, $(2, 4)$, $(2, 5)$, $(2, 6)$, $(3, 5)$, $(5, 1)$, $(6, 1)$, $(6, 4)$, $(6, 5)$, $(6, 7)$, $(7, 1)$, $(7, 2)$, $(7, 3)$, $(7, 4)$, $(7, 5)$, $(7, 6)$ & $(1, 2, 6)$, $(1, 2, 7)$, $(1, 3, 5)$, $(1, 4, 7)$, $(1, 6, 7)$, $(2, 3, 4)$, $(2, 3, 6)$, $(2, 5, 7)$, $(2, 6, 7)$, $(3, 6, 7)$, $(4, 5, 6)$, $(5, 6, 7)$ & $\Fib \boxtimes \PSU(2)_5$\cite{RSW,B1}\\ \hline 

$1$, $4$, $5$, $6$, $7$ & $(1, 4)$, $(1, 5)$, $(1, 6)$, $(1, 7)$, $(2, 4)$, $(2, 5)$, $(2, 7)$, $(3, 7)$, $(4, 1)$, $(4, 2)$, $(4, 5)$, $(4, 6)$, $(4, 7)$, $(5, 4)$, $(5, 6)$, $(5, 7)$, $(6, 1)$, $(6, 2)$, $(6, 3)$, $(6, 4)$, $(6, 5)$, $(6, 7)$, $(7, 5)$ & $(1, 2, 4)$, $(1, 2, 5)$, $(1, 2, 6)$, $(1, 2, 7)$, $(1, 3, 4)$, $(1, 3, 5)$, $(1, 4, 5)$, $(1, 4, 6)$, $(1, 5, 6)$, $(1, 6, 7)$, $(2, 3, 5)$, $(2, 3, 7)$, $(2, 4, 6)$, $(2, 5, 6)$, $(3, 4, 6)$, $(4, 5, 6)$, $(4, 5, 7)$, $(4, 6, 7)$ & \\ \hline 

$1$, $2$, $3$, $4$, $5$, $7$ & $(1, 6)$, $(2, 6)$, $(3, 6)$, $(4, 6)$, $(5, 1)$, $(5, 2)$, $(5, 3)$, $(5, 4)$, $(5, 7)$, $(7, 1)$, $(7, 2)$, $(7, 3)$, $(7, 4)$, $(7, 5)$ & $(1, 2, 3)$, $(1, 2, 4)$, $(1, 3, 4)$, $(1, 5, 7)$, $(2, 3, 4)$, $(2, 5, 7)$, $(3, 5, 7)$, $(4, 5, 7)$, $(5, 6, 7)$ & \\ \hline 

$1$, $2$, $3$, $4$, $5$, $7$ & $(1, 5)$, $(1, 6)$, $(1, 7)$, $(2, 5)$, $(2, 6)$, $(2, 7)$, $(3, 5)$, $(3, 6)$, $(3, 7)$, $(4, 6)$, $(5, 4)$, $(5, 7)$, $(7, 4)$, $(7, 5)$ & $(1, 2, 3)$, $(1, 2, 4)$, $(1, 2, 5)$, $(1, 2, 7)$, $(1, 3, 4)$, $(1, 3, 5)$, $(1, 3, 7)$, $(2, 3, 4)$, $(2, 3, 5)$, $(2, 3, 7)$, $(4, 5, 7)$, $(5, 6, 7)$ & \\ \hline 

$1$, $3$, $4$, $5$, $6$, $7$ & $(1, 4)$, $(1, 7)$, $(3, 1)$, $(3, 2)$, $(3, 5)$, $(3, 7)$, $(4, 2)$, $(5, 1)$, $(5, 2)$, $(5, 6)$, $(5, 7)$, $(6, 1)$, $(6, 2)$, $(6, 3)$, $(6, 7)$, $(7, 1)$, $(7, 4)$ & $(1, 2, 7)$, $(1, 3, 5)$, $(1, 3, 6)$, $(1, 4, 7)$, $(1, 5, 6)$, $(3, 4, 5)$, $(3, 4, 6)$, $(3, 5, 6)$, $(3, 5, 7)$, $(3, 6, 7)$, $(4, 5, 6)$, $(5, 6, 7)$ & \\ \hline 

$1$, $2$, $4$, $5$, $6$, $7$ & $(1, 4)$, $(1, 5)$, $(1, 6)$, $(2, 4)$, $(2, 5)$, $(2, 6)$, $(4, 1)$, $(4, 2)$, $(4, 5)$, $(4, 6)$, $(4, 7)$, $(5, 1)$, $(5, 4)$, $(5, 6)$, $(6, 2)$, $(6, 4)$, $(6, 5)$, $(7, 1)$, $(7, 2)$, $(7, 4)$, $(7, 5)$, $(7, 6)$ & $(1, 2, 4)$, $(1, 2, 5)$, $(1, 2, 6)$, $(1, 2, 7)$, $(1, 3, 5)$, $(1, 4, 6)$, $(1, 4, 7)$, $(1, 5, 7)$, $(1, 6, 7)$, $(2, 3, 6)$, $(2, 4, 5)$, $(2, 4, 7)$, $(2, 5, 7)$, $(2, 6, 7)$, $(3, 4, 7)$, $(4, 5, 6)$, $(4, 5, 7)$, $(4, 6, 7)$, $(5, 6, 7)$ & \\ \hline 

$1$, $3$, $4$, $5$, $6$, $7$ & $(1, 3)$, $(1, 4)$, $(1, 5)$, $(1, 6)$, $(1, 7)$, $(3, 1)$, $(3, 2)$, $(3, 5)$, $(3, 7)$, $(4, 2)$, $(5, 1)$, $(5, 2)$, $(5, 6)$, $(5, 7)$, $(6, 1)$, $(6, 2)$, $(6, 3)$, $(6, 7)$, $(7, 1)$, $(7, 3)$, $(7, 4)$, $(7, 5)$, $(7, 6)$ & $(1, 2, 7)$, $(1, 3, 5)$, $(1, 3, 6)$, $(1, 3, 7)$, $(1, 4, 7)$, $(1, 5, 6)$, $(1, 5, 7)$, $(1, 6, 7)$, $(3, 4, 5)$, $(3, 4, 6)$, $(3, 5, 6)$, $(3, 5, 7)$, $(3, 6, 7)$, $(4, 5, 6)$, $(5, 6, 7)$ & \\ \hline 

$1$, $2$, $3$, $4$, $5$, $6$, $7$ & $(1, 4)$, $(1, 5)$, $(2, 4)$, $(2, 6)$, $(3, 5)$, $(3, 6)$, $(7, 1)$, $(7, 2)$, $(7, 3)$, $(7, 4)$, $(7, 5)$, $(7, 6)$ & $(1, 2, 3)$, $(1, 2, 7)$, $(1, 3, 7)$, $(1, 4, 5)$, $(1, 6, 7)$, $(2, 3, 7)$, $(2, 4, 6)$, $(2, 5, 7)$, $(3, 4, 7)$, $(3, 5, 6)$ & $\Fib^3$ \cite{RSW} \\ \hline 

$1$, $2$, $3$, $4$, $5$, $6$, $7$ & $(1, 2)$, $(1, 3)$, $(1, 4)$, $(1, 5)$, $(2, 1)$, $(2, 3)$, $(2, 4)$, $(2, 6)$, $(3, 1)$, $(3, 2)$, $(3, 5)$, $(3, 6)$, $(4, 1)$, $(4, 3)$, $(4, 5)$, $(4, 6)$, $(5, 2)$, $(5, 3)$, $(5, 4)$, $(5, 6)$, $(6, 1)$, $(6, 2)$, $(6, 4)$, $(6, 5)$, $(7, 1)$, $(7, 2)$, $(7, 3)$, $(7, 4)$, $(7, 5)$, $(7, 6)$ & $(1, 2, 4)$, $(1, 2, 5)$, $(1, 2, 7)$, $(1, 3, 5)$, $(1, 3, 6)$, $(1, 3, 7)$, $(1, 4, 6)$, $(1, 4, 7)$, $(1, 5, 6)$, $(1, 5, 7)$, $(1, 6, 7)$, $(2, 3, 4)$, $(2, 3, 6)$, $(2, 3, 7)$, $(2, 4, 5)$, $(2, 4, 7)$, $(2, 5, 6)$, $(2, 5, 7)$, $(2, 6, 7)$, $(3, 4, 5)$, $(3, 4, 6)$, $(3, 4, 7)$, $(3, 5, 7)$, $(3, 6, 7)$, $(4, 5, 7)$, $(4, 6, 7)$, $(5, 6, 7)$ & \\ \hline 

\end{longtable}

\bibliographystyle{siam}
\bibliography{fusion}

\end{document}